\documentclass[english,final]{elsarticle}

\usepackage[utf8]{inputenc}
\DeclareUnicodeCharacter{00A0}{}

\usepackage[a4paper]{geometry}

\usepackage[T1]{fontenc}
\usepackage{lmodern}

\usepackage{multirow}

\usepackage{graphicx}
\usepackage{float}
\usepackage{wrapfig}

\usepackage{anysize}

\usepackage{enumerate}

\usepackage{cleveref}

\usepackage{amsmath,amsfonts,amsthm,amssymb}
\usepackage{mathtools}

\usepackage{xcolor}

\graphicspath{images}

\usepackage{xypic}

\theoremstyle{definition}
\newtheorem{defn}{Definition}[section]

\newtheorem{assumption}[defn]{Assumption}
\newenvironment{hyp}    
{%
	\pushQED{\qed}\begin{assumption}}
	{\popQED\end{assumption}}

\newtheorem{rem}[defn]{Remark}

\theoremstyle{plain}
\newtheorem{thrm}{Theorem}
\newtheorem{lem}[defn]{Lemma}
\newtheorem{prop}[defn]{Proposition}

\newcommand{\partiald}[2]{\frac{\partial #1}{\partial #2}}
\newcommand{\der}[2]{\frac{{\rm d} #1}{{\rm d} #2}}

\newcommand{\Li}{\mathrm{Li}}
\newcommand{\tend}{{t_{\mathrm{end}}}}
\newcommand{\tendI}{{t_{\mathrm{end}}^I}}
\newcommand{\J}{J_\delta}
\newcommand{\dx}{\,\mathrm{d}x}
\newcommand{\dt}{\,\mathrm{d}t}
\newcommand{\ds}{\,\mathrm{d}s}
\newcommand{\dr}{\,\mathrm{d}r}

\newcommand{\ce}{c_{\mathrm{e}}}
\newcommand{\cs}{c_{\mathrm{s}}}
\newcommand{\csB}{c_{\mathrm{s,B}}}
\newcommand{\csmax}{c_{\mathrm{s,max}}}
\newcommand{\phie}{\varphi_{\mathrm{e}}}
\newcommand{\phieLi}{\varphi_{\mathrm{e},\Li}}
\newcommand{\phis}{\phi_{\mathrm{s}}}
\newcommand{\Tamb}{T_{\mathrm{amb}}}
\newcommand{\phistest}{\psi_{\mathrm{s}}}
\newcommand{\phietest}{\psi_{\mathrm{e}}}
\newcommand{\cstest}{\Psi_{\mathrm{s}}}
\newcommand{\XPhi}{X_{\phi}}
\newcommand{\Rs}{R_{\mathrm{s}}}
\newcommand{\Rsp}{R_{\mathrm{s},+}}
\newcommand{\Rsm}{R_{\mathrm{s},-}}

\newcommand{\Done}{{{D_\delta}}}
\newcommand{\Dthree}{D_\delta^{3D}}

\newcommand{\alphaceLi}{{\alpha_{\rm a}}}
\newcommand{\alphacsLi}{{\alpha_{\rm s}}}
\newcommand{\alphacsLip}{{\alpha_{\mathrm {s},+}}}
\newcommand{\alphacsLim}{{\alpha_{\mathrm {s},-}}}

\newcommand{\alphaceeq}{{\alpha_{\rm e}}}
\newcommand{\alphacseq}{{\alpha_{\rm s}}}

\newcommand{\betacs}{{\beta_{\rm a}}}

    \makeatletter
    \def\ps@pprintTitle{%
    	\let\@oddhead\@empty
    	\let\@evenhead\@empty
    	\let\@oddfoot\@empty
    	\let\@evenfoot\@oddfoot
    }
    \makeatother

\begin{document}

	\begin{frontmatter} 		
		\title{On the well-posedness of a multiscale {mathematical} model for Lithium-ion batteries }
		
				\date{\today}
		
		\author[imi,appmath]{J.I. Díaz}
		\ead{jidiaz@ucm.es}
		
		\address[imi]{Instituto de Matemática Interdisciplinar,
			Universidad Complutense de Madrid,\\
			Plaza de Ciencias, 3, 28040 Madrid (Spain),}
		
		\address[appmath]{Dept. of Applied Mathematics and Mathematical Analysis,  
			Universidad Complutense de Madrid,\protect\\
			Plaza de Ciencias, 3, 28040 Madrid (Spain),\protect\\}
		
		\author[imi,icai]{D. Gómez-Castro}
		\ead{dgcastro@ucm.es}
		
		\address[icai]{Departamento de Matemática Aplicada, 
			Escuela Técnica Superior de Ingeniería - ICAI,\protect\\
			Universidad Pontificia Comillas, 
			C/ Alberto Aguilera, 25, 28015 Madrid (Spain)} 
		
		\author[imi,appmath]{A.M. Ramos}
		\ead{angel@mat.ucm.es}
				\begin{abstract}
					We consider the mathematical treatment of a system of nonlinear partial differential equations based on a model, proposed in 1972 by J. Newman, in which the coupling between the Lithium concentration, the phase potentials and temperature in the electrodes and the electrolyte of a Lithium battery cell is considered.  After introducing some functional space{s} well-adapted to our framework we obtain some rigorous results showing the well-posedness of the system, first for some short time and then, by considering some hypothesis on the nonlinearities, globally in time. As far as we know, this is the first result in the literature proving existence in time of the full Newman model, which follows previous results by the third author in 2016 regarding a simplified case. 
				\end{abstract}
				
	\end{frontmatter}
	
	\section{Introduction}

	Let us suppose we have a mathematical system of differential equations equations involving, {after a} homogenization {process} (or other methods), two different scales, that we may call macro and micro, for simplicity. A macro-scale domain is given by $x\in (0,L)$, where different processes may occur in three relevant subintervals: $(0,L_1)$, $(L_1,L_1+\delta)$ and $(L_1+\delta,L)$. At each point {$x\in (0,L_1) \cup (L_1+\delta, L)$} we consider that there is a micro scale domain given by a sphere with radius $R(x)>0$. Different processes are modeled in each scale with differential equations, which are coupled at different levels, including boundary conditions.
	
	Let us considered, motivated by the modeling of charge transport in Lithium batteries (as we will show below) the following (relatively) abstract framework (which can be generalized to other cases).
	
	Our system has 5 unknowns functions: $u(x,t)$, $v(x;r,t)$, $\varphi(x,t)$, $\phi(x,t)$ and $\theta (t)$ such that
	\begin{itemize}
		\item $u:(0,L)\times (0,t_{\rm end})\rightarrow \mathbb{R}$,
		\item $v:\Omega_\delta \times (0,t_{\rm end})\rightarrow \mathbb{R}$,
		\item $\varphi:(0,L)\times (0,t_{\rm end})\rightarrow \mathbb{R}$,
		\item $\phi:(0,L_1)\cup (L_1+\delta ,L)\times (0,t_{\rm end})\rightarrow \mathbb{R}$,
		\item $\theta: (0,t_{\rm end})\rightarrow \mathbb{R}$,
	\end{itemize}
	where $t_{\rm end}>0$ and
	$
	\Omega_\delta =\{ (x,r): x\in (0,L_1)\cup (L_1+\delta ,L) \mbox{ and } r\in [0, R(x)]\}.
	$
	Those functions satisfy the following system of macro scale equations: 
	\begin{equation} \label{gsoe1}
	\begin{dcases}
	\varepsilon \frac{\partial u}{\partial t} -{\mathcal L}_1 u =F_1(x,u,v,\varphi,\phi,\theta) & \mbox{in } (0,L)\times (0,t_{\rm end}), \\
		\mbox{with suitable boundary and initial conditions, and, for each }t \in (0,t_{\rm end}): &  \\
	\quad  {\displaystyle -{\mathcal L}_2 \varphi + {\mathcal L}_3 f(u)=F_3(x,u,v,\varphi,\phi,\theta)} & \mbox{in } (0,L), \\
	\quad  -{\mathcal L}_4 \phi =F_4(x,u,v,\varphi,\phi,\theta) & \mbox{in } (0,L_1)\cup (L_1+\delta L), \\
	\quad \mbox{with suitable boundary conditions.}
	\end{dcases}
	\end{equation}
	This system is coupled with the following system of  micro scale diffusion equation:
	\begin{equation} \label{gsoe2}
	\left\{
	\begin{array}{ll}
	\mbox{For almost each }x \in (0,L_1)\cup (L_1+\delta ,L): \\
	\ \ {\displaystyle \frac{\partial v}{\partial t} -\frac{1}{r^2}\frac{\partial}{\partial r}\left( r^2 D_2\frac{\partial v}{\partial r}\right) }=0 & \mbox{ in } (0,R(x))\times (0,t_{\rm end}) \\
	\ \ {\displaystyle \frac{\partial v}{\partial r} (x;0,\cdot)=0}  & \mbox{ in } (0,t_{\rm end}) \\
	\ \ {\displaystyle-D_2\frac{\partial v}{\partial r} (x;R(x),\cdot)=F_2(x,u,v,\varphi,\phi,\theta)}  & \mbox{ in } (0,t_{\rm end}) \\
	\ \ v(\cdot , 0)=v_{0} & \mbox{ in }(0,L).
	\end{array}
	\right.
	\end{equation}
	In the equations written above  ${\mathcal L}_1$ to ${\mathcal L}_4$ are second order differential operators (which may also depend on $x$, $u$, $v$, $\varphi$, $\phi$ and $\theta$), $\varepsilon$ is a function depending on $x\in (0,L)$ and $f$, $F_1$, $F_2$, $F_3$ and $F_4$ are real-valued functions.
	
	Finally, we add the following initial value problem:
	\begin{equation} \label{gsoe3}
	\begin{dcases} 
	\theta'(t) =  f(t,\theta (t);u, v, \varphi, \phi) \\
	\theta (0)=\theta_{0} ,
	\end{dcases}
	\end{equation}
	where the dependence of $f$ on $u, v, \varphi, \phi$ may be global in space.

	The question that arises is the following. Under which conditions can we prove existence and/or uniqueness of a local (in time) solution $(u,v,\varphi,\phi,\theta)$ of equations (\ref{gsoe1})--(\ref{gsoe3}) and under which conditions is it global in time?
	\\
	
	We study here a particular interesting case arising in the modeling of Lithium batteries. Lithium-ion batteries are currently extensively used for storing electricity in mobile devices, from phones to cars. Nevertheless they are still many drawbacks for them, as their reduced charge capacity, the long time needed to charge them, thermal runaways, etc. Therefore, a lot of research is being done in order to improve these devices. A good mathematical model of the transport mechanisms within batteries is very important in that research in order to understand the physics involved in the processes and to allow quick numerical experiments. But this models need to be well understood from a mathematical point {of view}, so that conclusions extracted from them are more rigorous. The well-posedness of the mathematical models being used is, therefore, one of the key points to be considered.  This is, indeed, the goal of this paper.
	\\

	In \cite{Ramos:2015:lithiumionbatteries} a full model for Lithium ion batteries was presented, based on on the well known J. Newman model (see \cite{newman1972electrochemical}), and partial results for the well posedness were given. In this paper, we intend to complete those well-posedness results and study the regularity of the solutions. Let us write the complete system as  presented in \cite{Ramos:2015:lithiumionbatteries} but using in the electrolyte the electric potential measured by a reference Lithium electrode ({$\phie$}) instead of its real electric potential. This is done because in electrochemical applications the potential in an electrolyte is typically measured by inserting a reference electrode of a pure compound, tipically a Lithium electrode (see {\cite{Bizeray2016,Ranon2014,richardson+ramos+ranon+please-charge+transport+lithium}}). {Local existence means that we are able to prove existence of solutions if the final time $\tend$ is ``small enough''}. In a special case, we will show that $\tend$ can be taken as large as wanted, making the existence of solutions global in time. \\
	
	A typical Lithium-ion battery cell has three regions: a porous negative electrode, a porous positive electrode and an electro-blocking separator. Furthermore, the cell contains an electrolyte, which is a concentrated solution containing charged species that move along the cell in response to an electrochemical potential gradient.\\	
	
	Let $L$ be the length of a cell of a battery, $L_1$ be the length of the negative electrode and $\delta$ the length of the separator. We assume the radius of an electrode particle to be $\Rs (x)$, with
	\begin{equation} \label{eq:R s}
		\Rs(x) = \begin{dcases}
			\Rsm & \text{ if } x \in (0,L_1), \\
			\Rsp & \text{ if } x \in (L_1 + \delta, L).
		\end{dcases}
	\end{equation}
	A schematic representation of a cell is given in Figure \ref{fig:schematic battery}.
	
	\begin{figure}[H]
		\centering
		
		\includegraphics[]{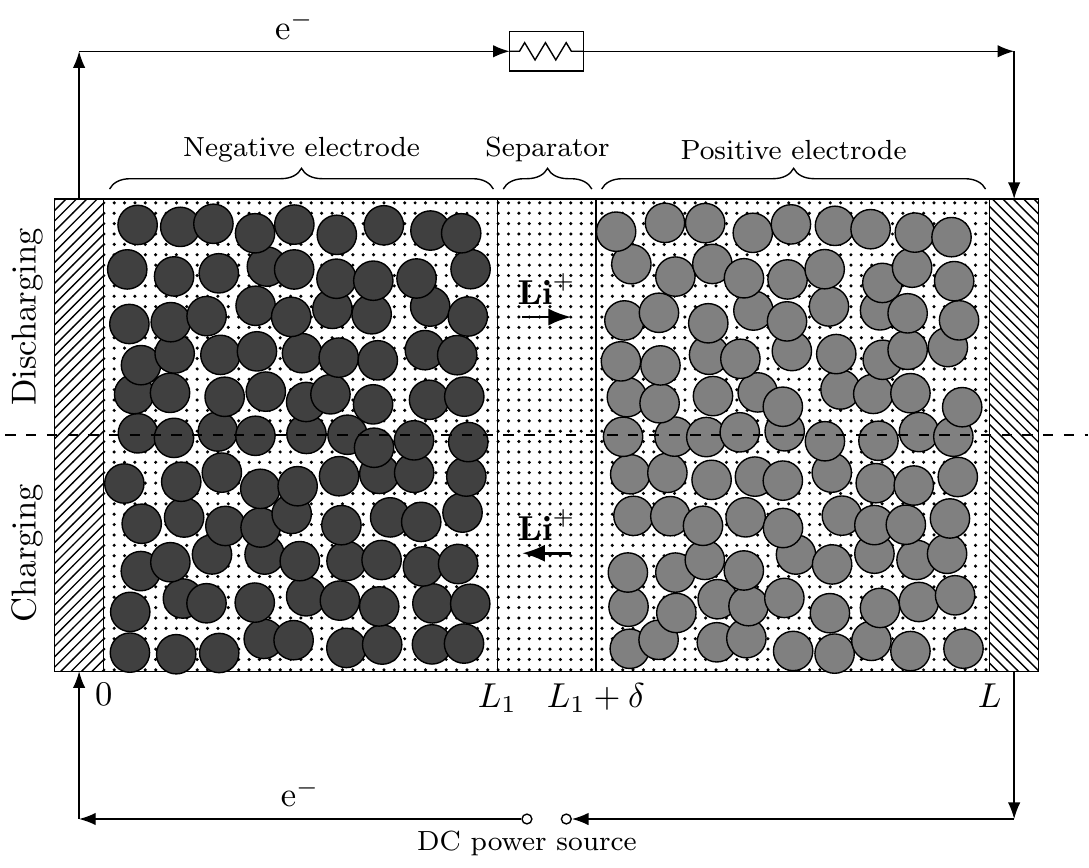}

		\caption{Schematic representation of a battery. Dark gray disks represent negative electrode particles, light gray disks represent positive electrode particles, the dotted region represents the presence of electrolyte. On both sides of the cell the lined region represents the current collectors.}
		
		\label{fig:schematic battery}
	\end{figure}
	The unknowns in the system of equations that we study are:
	\begin{itemize}
		\item {The concentration of Lithium ions in the electrolyte} $\ce = \ce (x,t)$ in every macroscopic point $x \in(0,L)${.}
		\item {The concentration of Lithium ions in the electrodes} $\cs = \cs (r,t;x)$, which for every macroscopic point $x \in (0,L_1) \cup (L_1+\delta, L)$ is defined in a microscopic ball {of radius $\Rs(x)$}, with $r$ indicating the distance to the center. We assume  radial symmetry in the diffusion taking place in each particle. Therefore, we do not need to consider angular coordinates{.}
		\item The electric potential $\phis = \phis(x,t)$ in the electrodes{.}
		\item  The electric potential measured by a reference Lithium electrode in the electrolyte, $\phie = \phie(x,t)$.
		\item The temperature $T(t)$ in the cell{.}
	\end{itemize}	
		
	{The system of equations is given by \eqref{eq:model ce}--\eqref{eq:model T} below}. \\
	
	The transport of Lithium through the electrolyte can be modeled by the following macro-scale system of equations {(conservation of Lithium in the electrolyte)}:
	\begin{align}
		&\label{eq:model ce}\begin{dcases} 
			\partiald {\ce} t - \frac{\partial }{\partial x} \left( D_{\textrm e} \frac{\partial \ce}{\partial x} \right)  = 	
			\alphaceeq {j^{\Li}}, & \textrm{in } (0,L) \times (0,\tend), \\
			\partiald {\ce} {x} (0,t) = \partiald {\ce} {x} (L,t) = 0, & t\in (0,\tend),\\
			\ce(x,0) = c_{\rm e,0}(x), & x \in (0,L),
		\end{dcases}
	\end{align}
	where $j^\Li {=j^\Li(x,\ce(x,t), \cs (\Rs(x),t{;x}), \phie(x,t), \phis(x,t), T(t)) }$ {is the reaction current resulting from intercalation on $\Li$ into solid electrode particles and} represents the Lithium flux between the electrodes and the electrolyte (which is a nonlinear function of all the variables, an expression {that} will be given later),
	$D_{\textrm e} > 0$ is a diffusion {function}, $\alphaceeq > 0$ is {constant} and $c_{\rm e, 0}$ is a known initial state.
	\\
	
	The transport of Lithium through the electrodes can be modeled, at almost every macroscopic point $x \in (0,L_1) \cup (L_1 + \delta, L)$, by the following micro-scale system of equations {(conservation of Lithium in the electrodes)}, {written in radial coordinates due to the radial symmetry}:
	\begin{align}
		& \label{eq:model cs}
		\begin{dcases}
			\partiald {\cs} t - \frac{D_s}{r^2} \partiald {} r \left( r^2 \partiald {\cs} r \right)  = 	
			0, &  \textrm{in }  (0,\Rs(x)) \times (0,\tend),\\
			\partiald {\cs} {r} (0,t;x) = 0, & {t\in (0,\tend),}\\
			\quad -D_s \partiald {\cs} {r} (\Rs(x),t;x) = \alpha_s(x) {j^{\Li}} & t\in (0,\tend),\\
			\cs(r,0;x) = c_{s,0}(r;x), & r \in (0,\Rs(x)),
		\end{dcases}
	\end{align}
	where $D_s > 0$ is a diffusion coefficient {and} $\alpha_s > 0${,} {both with a constant value} in $(0,L_1)$ {and another constant in} $(L_1+\delta, L)${. Furthermore} $c_{\rm s, 0}$ is a radially symmetric initial state. Notice that problem \eqref{eq:model cs} is written in spherical coordinates.\\

	The electric potential in the electrodes, $\phis$, and the electric potential measured by a reference Lithium electrode in the electrolyte, $\phie$, can be modelled, for each $t \in (0,\tend)$, by the following {equations {expressing} the conservation of charge in the electrolyte}:
	\begin{align}
		& \label{eq:model phie}\begin{dcases}
			- \partiald{}{x} \left( \kappa \partiald{\phie}{x}  \right) + \alpha _{\phie} T  \partiald {} x \left( \kappa \partiald{} {x} (f_{\phie} (\ce)) \right)   = {j^{\Li}} , & \textrm{in } (0,L),\\
			\partiald {\phie} {x} (0,t) = \partiald {\phie} {x} (L,t) = 0.\\
		\end{dcases}\\
		\intertext{{and the conservation of charge in the electrodes}}
		&\label{eq:model phis}\begin{dcases}
			-  \partiald{}{x} \left( \sigma \partiald{\phis}{x}  \right) = - {j^{\Li}} , & \textrm{in } (0,L_1) \cup (L_1 + \delta, L),\\
			\sigma(0) \partiald {\phis} {x} (0,t) = \sigma(L) \partiald {\phis} {x} (L,t) = -\frac {I(t)}{{ A }}, \\
			\quad \partiald {\phis} x (L_1,t) = \partiald {\phis} x (L_1 + \delta, t) = 0 ,
		\end{dcases}
		\end{align}
	where $\kappa = \kappa(\ce, T)$, given by a function $\kappa \in \mathcal C^2 ((0,+\infty)^2)$, $\kappa > 0$, and $\sigma \in L^\infty ({(}0,L_1{)} \cup {(}L_1 + \delta, L{)})$ are conductivity coefficients (uniformly positive), $\alpha _{\phie} \ge 0$ is a constant, $f_{\phie} \in \mathcal C^2 (0,+\infty)$, {$A$ is the cross-sectional area (also the currect collector area)} and $I$, the input current, is a piecewise constant function defined for $t \in [0,\tendI]$. {Naturally, if the input current is defined up to a time $\tendI$, then we can only expect to solve the system up to a time $\tend \le \tendI$.}
	\\
		
	Finally, we consider the temperature $T = T(t)$ inside the battery, which we consider spatially constant. Its time evolution is given by
		\begin{align}		
		& \label{eq:model T} 
		\begin{dcases}
			\frac{{\rm d} T}{{\rm d} t} (t) = - \alpha_T (T(t) - \Tamb) + F_T (\ce(\cdot, t), \cs (\Rs(\cdot) , t{; \cdot}) , \phie(\cdot, {t}), \phis(\cdot, {t}), T(t)), & t \in (0, \tend), \\
			T(0) = T_0,
		\end{dcases}
	\end{align} 
	where $\alpha_{T}$ is a positive constant, $\Tamb$ represents the ambient temperature, $T_0 \ge 0$ is an initial known temperature and $F_T$ is a continuous functional that represents the effect of the other variables over $T$ (an expression {that} will be given later). $F_T$ may depend of the values at time $t$ of functions $\ce, \cs, \phie$ and $\phis$ {locally or globally in space} (see \eqref{eq:defn FT start} - \eqref{eq:defn FT end}). \\
	
	In \cite{Ramos:2015:lithiumionbatteries}, {instead of} equations \eqref{eq:model ce}, \eqref{eq:model phie} and \eqref{eq:model phis} {the {author} consider the equations}
	\begin{align*}
	\varepsilon_{\rm e}\partiald {\ce} t & - \frac{\partial }{\partial x} \left( D_{\textrm e}\varepsilon_{\rm e}^{p} \frac{\partial \ce}{\partial x} \right)  = 	
	\alpha_e j^{\Li} ,  \quad \textrm{in } (0,L) \times (0,\tend),\\
	&- \partiald{}{x} \left( \varepsilon_{\rm e}^{p} \kappa \partiald{\phie}{x}  \right) + \alpha _{\phie} T  \partiald {} x \left( \varepsilon_{\rm e}^{p}  \kappa \partiald{} {x} (f_{\phie} (\ce)) \right)   = j^{\Li}  , \quad   \textrm{in } (0,L),\\
	&-\varepsilon_s \sigma \frac{\partial \phis}{\partial x^2} = - j^\Li , \quad \textrm{in } (0,L_1) \cup (L_1+ \delta, L).
	\end{align*}
	where $\varepsilon_{\rm e} > 0$ is constant in $(0,L_1), (L_1 , L_1 + \delta ) , (L_1 + \delta, L)$ {and $\varepsilon_{\rm s}$ in $(0,L_1),  (L_1 + \delta, L)$}. This general case{, which is not in divergence form for $\ce$ and $\phie$, } introduces an extra level of difficulty we will not consider here. The same techniques we  {present} in this paper apply to that case, but with the introduction of some additional technicalities (see, e.g., \cite{diaz+hetzer1998,bermejo+carpio+diaz+tello2009} and the references therein). 
	
	Here we have considered $\varepsilon_{\rm e}$ constant in $(0,L)$, and therefore we can remove it by assimilating $\varepsilon_{\rm e}^{p-1}$ in $D_{\rm e}$, $\varepsilon_{\rm e}^{-1}$ in $\alphaceeq$ and $\varepsilon_{\rm e}^p$ in $\kappa$.\\
	
	We simplify the domain notation by introducing the following sets:
	{
	\begin{align}
		 \label{eq:defn Ics} \J & =  (0,L_1) \cup (L_1 + \delta , L) \\
		 \Done &= \bigcup_{x\in J_\delta} \{x \} \times [{0,\Rs ({x})}] .
	\end{align}
	}
	In fact, we will consider $\Rs$ constant in both $(0,L_1)$ and $(L_1 + \delta, L)$ {as in \eqref{eq:R s}} (see Figure \ref{fig:Jdelta Ddelta}).
	\begin{figure}[H]
		\centering
		\includegraphics[scale = 1]{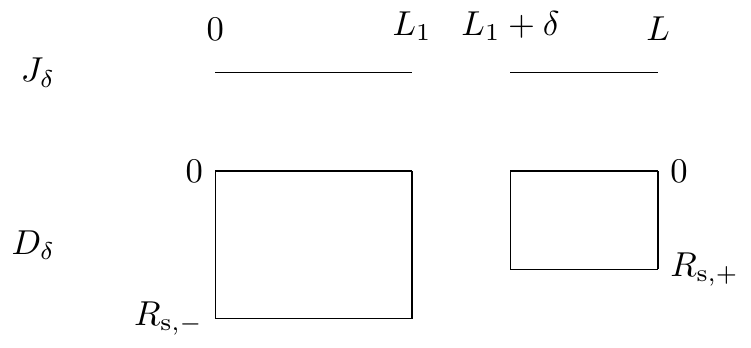}
		\caption{Domains $J_\delta$ (spatial domain of definition of $\ce, \phie, \phis$) and {$\Done$} (spatial domain of definition of $\cs$). Notice that, since we are using a radial {coordinate}, every {segment} $\{x \} \times [{0,\Rs ({x})}]$ {in $\Done$} represents a ball $ \{x \} \times B_{\Rs ({x})}$ {in $\{x\} \times \mathbb R^3$}.}
		\label{fig:Jdelta Ddelta}
	\end{figure}

	The analytical expression of $j^\Li$ is given as follows, for $ ( \ce,  \cs,  \phie,  \phis ,  T) \in \mathbb R^5 $:
	\begin{align}	
		\label{eq:j Li}
		j^{\Li}(x,  \ce,  \cs,  \phie,  \phis ,  T) &= \begin{dcases} 
			\overline{j}^\Li (x,  \ce,  \cs,  T, \eta (x, \ce,  \cs,  \phie,  \phis, T) ) , & x \in \J , \\
			0, &  \textrm {otherwise}, 
			\end{dcases}\\
		 \eta (x, \ce,  \cs,  \phie,  \phis,  T) & =  \phis -  \phie - U(x,  \ce,  \cs,  T), \quad x \in \J,
	\end{align}
	where $U$ is the open circuit potential and $\eta$ the surface overpotential of the corresponding electrode reaction.\\
	
	There is no common agreement on the structural assumptions of $U$. Some papers {use} a fitting {function} either polynomial  \cite{smith+wang2006diffusion+limitation+lithium} or exponential \cite{Birkl2015}, whereas other authors propose {an improved version with} logarithmic behaviour {close to the limit cases} \cite{ramos+please2015arxiv,richardson+ramos+ranon+please-charge+transport+lithium}. {We will start working in a general framework, which contains as a particular } important example the Butler-Volmer flux (see { \cite{newman1972electrochemical,Ramos:2015:lithiumionbatteries,richardson+ramos+ranon+please-charge+transport+lithium}}), in which the functions are taken as
		\begin{align} 
		\label{eq:Butler-Volmer start}
		\overline j^{\Li} &=  \ce ^{\alpha_a}  \cs^{{\alpha_c}} (\csmax -  \cs )^{{ \alpha_a} }  {h \left (x,  \frac{1}{T}  \eta \right )} \\
		{ h ( x, \eta) } &= {\delta_1  (x)} \exp (\alpha_a  \eta ) - 	{\delta_2 (x)} \exp ( - \alpha_c  \eta ) ,\\
		f_{\phie} ( \ce) &= \ln  {\ce} ,\\
		\alpha_a & \in (0,1), \\
		\alpha_c & \in (0,1) ,
		\label{eq:Butler-Volmer end}
		\end{align}
		where {{$\delta_1$, $\delta_2$ are positive and} constant in each electrode,} $\csmax$ is a constant that represents the maximum value of $\cs$, $\alpha_a$ and $\alpha_c$ (dimensionless constants) are anodic and cathodic coefficients, respectively, for an electrode reaction.
		
		Function $U$ was later  proposed {in} \cite{ramos+please2015arxiv} as 
		\begin{align}
		\label{eq:open circuit potential}
		U &= - \alpha (x)  T \ln \cs + \beta (x)  T \ln (\csmax -  \cs) + \gamma (x)  T \ln  \ce + p ( \ce,  \cs ,  T),
		\end{align}
		where $\alpha, \beta, \gamma$ are positive functions in $L^\infty (\J)$ and $p$ is a smooth bounded function.\\
		
		The following particular choice of $F_T$  {can be} considered {(see, e.g., \cite{Ramos:2015:lithiumionbatteries})}:
		\begin{align}
			F_T &=q_r + q_j + q_c + q_e , \label{eq:defn FT start}\\
			q_r &= A \int_0^L j^\Li \eta \dx, \\
			q_j &= A \int_{\J} \sigma \left( \frac{\partial \phis}{\partial x} \right)^2 + A \int_0^L \left[ \kappa \left(  \frac{\partial \phie}{\partial x}  \right)^2 + {\alphacseq} T \kappa \left( \frac{\partial \ln \ce }{\partial x} \right) \left( \frac{\partial \phie}{\partial x} \right) \right ]\dx, \\
			q_c &={ \frac{R_{\rm f}}{A} } I(t)^2, \\
			q_e &= TA \int_{\J} \left[ j^\Li  \frac{\partial U}{\partial T} \left( \frac {\cs (\Rs (x) , t{; x})}{\csmax}  \right) \right]\dx{,} \label{eq:defn FT end}
		\end{align}
		where $\ln$ is the natural logarithm {and $R_{\rm f}$ is the film resistance of the electrodes.} 
		\\
	
		We point out that $T$ has been assumed to be uniform across the whole cell. Nonetheless, a more complete model (specially one dealing with temperature blow up) {with} a heat diffusion equation could be {used}. The study of blow up in {this type of} equations has been largely studied (see, for example, \cite{diaz+antontsev2012energy+methods,diaz+casal+vegas2009blow+up+delay}).\\

	Using this model the voltage {at time $t$} of the cell can be estimated {(see \cite{Ramos:2015:lithiumionbatteries}) as
	\begin{equation} \label{eq:voltage}
		V(t) = \phis(L,t)  - \phis(0,t) - \frac{R_{\rm f}}{A} I. \\
	\end{equation}}

	{The model, as written above, presents a number of difficulties: pseudo-two dimensional (P2D) model, elliptic equations coupled with parabolic equations, discontinuities in space of some of the involved functions, {the fact that $j^\Li$ is not smooth and not monotone}, etc. As far as we know {there is} no proof in the literature of the global existence of solution before this work.}
	\\

	In order to state the results, we will introduce the following definitions, which we introduce as a mathematical tool:
	\begin{align}
		\csB(x,t) &= \cs (\Rs (x), t; x) , \\
		\phieLi (x,t) &= \phie(x,t) - \alpha_{\phie} T (t) f_{\phie} (\ce (x,t)) \label{eq:defn phieLi}.
	\end{align}
	Notice that $\csB$ represents the only values of $\cs$ that affect $j^\Li$ and $\phieLi$ is the solution of
	\begin{equation} \label{eq:equation phieLi}
		\begin{dcases}
		-\partiald{}{x} \left( \kappa \partiald{\phieLi}{x}  \right)= j^\Li, & \textrm{{in} }(0,L){,} \\
		\kappa  \frac {\partial \phieLi}{\partial x} = 0, & \textrm{{at }}\{0,L\}{,} \\
		\end{dcases}
	\end{equation}
	which is {a} {system} simpler than \eqref{eq:model phie}.
	\\
	
	First we will state {four} general results for a large class of functions $j^\Li${, regarding the local in time existence of solutions of the general abstract problem and their maximal extension in time}. Although the detailed expression of the assumptions are only given in the next section the reader can be aware now of the different nature of our mathematical results.
	
	\begin{thrm}[Well-posedness {with general flux and nature of the possible blow-up}] \label{thrm:well posedness model}
		Let assumptions \ref{hyp:regularity data}, \ref{hyp:regularity of coefficients}, \ref{hyp:regularity flux general} and \ref{hyp:regularity FT} hold. Then, {if $\tend$ is small enough,} there exists a unique weak-mild solution {by parts} of \eqref{eq:model ce}-\eqref{eq:model T} (in the sense of Definition \ref{defn:solution general problem by parts} and satisfying Assumption \ref{hyp:constant for phieLi}). 
		
		Moveover, there exists a unique maximal extension defined for $t\in [0,\tend)$  where $\tend$ is some constant $\tend \le \tendI$. If $\tend < \tendI$ then one the following conditions holds as $t\nearrow  \tend$:
		\begin{align} \label{eq:solutions to boundary}
		\nonumber \min_{J_\delta \times [0,t] } \csB \to 0 & \quad \textrm{ or } \quad \max_{ J_\delta \times [0,t] } \csB \to \csmax \quad  \textrm{ or } \quad \min_{[0,L] \times [0,t]} \ce \to 0 \quad \textrm{ or } \quad \max_{[0,t] \times[0,L] } \ce \to +\infty \\ 
		&  \quad \textrm{ or } \quad\min_{[0,t]}  T \to 0 \quad  \textrm { or } \quad \max_{[0,t] } T  \to +\infty.
		\end{align}
	
	\end{thrm}

	We remark that without assumption \ref{hyp:constant for phieLi} existence for {potentials} $\phie, \phis$ can only be established up to a constant {(as expected)}.\\
	
	In the particular case of $j^\Li$ being the Butler-Volmer flux, given by \eqref{eq:j Li}-\eqref{eq:open circuit potential}, under some physically reasonable hypothesis on the parameters (see Assumptions \ref{hyp:constant coefficients}, \ref{hyp:csmax less 1}, \ref{hyp:flux ramos please}, \ref{hyp:bounds exponents} below), we will prove a general existence result, and characterize the nature of possible blow up. {It states that, under reasonable extra conditions, the non-physical obstructions for well-possedness $\csB \to 0, \csmax$ or $\ce \to 0,+\infty$ that appear in Theorem \ref{thrm:well posedness model}, are not the cause of the blow-up behaviour of local solutions.}
	
	\begin{thrm}[{Well-posedness for the Butler-Volmer flux and nature of the possible blow-up}] \label{thrm:blow up behaviour}
		Let Assumptions \ref{hyp:regularity data}, \ref{hyp:regularity of coefficients}, \ref{hyp:regularity FT}, \ref{hyp:constant coefficients},  \ref{hyp:csmax less 1}, \ref{hyp:flux ramos please} and \ref{hyp:bounds exponents} hold. 
		Then{, {if $\tend$ is small enough,}} there exists a unique weak-mild solution {by parts} of \eqref{eq:model ce}-\eqref{eq:model T} (in the sense of Definition \ref{defn:solution general problem by parts} and satisfying assumption \ref{hyp:constant for phieLi}). This solution admits a unique maximal extension in time with $\tend \le \tendI$. 
		
		Moveover, if $\tend < \tendI$ then, as $t \nearrow \tend$ either
		\begin{align} \label{eq:reasonable condition blowup}
		\max_{\J \times [0,t]} |\phis - \phieLi| \to + \infty \quad \textrm{ or } \quad \min_{[0,t]}  T \to 0 \quad \textrm{ or } \quad  \max_{[0,t]} T \to  + \infty. 
		\end{align}
		Furthermore,
		\begin{equation}
		0 < \cs < \csmax \qquad \textrm{and} \qquad \ce > 0, \qquad \forall\  0 \le t < \tend.
		\end{equation}
	\end{thrm}
	
	In the last part of {this} paper {(see Section \ref{sec:assumptions theorem 3 and 4})},  we give {an \emph{ad hoc}} modification of the system for which we can state a global existence theorem {by removing, in two steps, the impediments to global existence of solutions given by \eqref{eq:reasonable condition blowup}}. {Actually,} {since the blow-up conditions \eqref{eq:reasonable condition blowup} do not correspond to the physical intuition, it is likely that the solutions do not, in fact, develop this kind of behaviour.}\\
	
	In a first stage we will find a bound for the flux with respect to $\phis - \phieLi$,  and show:
	
	\begin{thrm}[Blow up behaviour when $j^\Li$ is bounded with respect to $\phis-\phieLi$] \label{prop:blow up truncated potential}
		Let Assumptions \ref{hyp:regularity data}, \ref{hyp:regularity of coefficients}, \ref{hyp:regularity FT}, \ref{hyp:constant coefficients},  \ref{hyp:csmax less 1}, {\ref{hyp:bounds exponents} and \ref{hyp:flux truncated}}   hold. Then, {if $\tend$ is small enough,} there exists a unique {weak-mild} solution {by parts} of \eqref{eq:model ce} - \eqref{eq:model T} (in the sense of Definition \ref{defn:solution general problem by parts} and satisfying Assumption \ref{hyp:constant for phieLi}). This solution admits a unique maximal extension in time with $\tend \le \tendI$. If $\tend < \tendI$, as $t \to \tend$ then either
		\begin{align} \label{eq:blow up of T}
		\min_{[0,t]}  T \to 0 \quad \textrm{ or } \quad  \max_{[0,t]} T \to  + \infty .
		\end{align}
	\end{thrm}
	Finally, by assuming some additional conditions (see Assumptions \ref{hyp:flux truncated}, \ref{hyp:cutoff FT}) we will obtain a bound for the temperature $T$ and prove what can be considered a first global existence result in the literature for this system:
	\begin{thrm}[Global existence in a modified case] \label{prop:global existence truncated}
			Let assumptions \ref{hyp:regularity data}, \ref{hyp:regularity of coefficients}, \ref{hyp:regularity FT}, \ref{hyp:constant coefficients}, \ref{hyp:csmax less 1}, {\ref{hyp:bounds exponents}, \ref{hyp:flux truncated}} and \ref{hyp:cutoff FT} hold. Then, there exists a unique {weak-mild} solution {by parts} of \eqref{eq:model ce} - \eqref{eq:model T}, defined for $t \in [0,\tendI]$ (in the sense of Definition \ref{defn:solution general problem} {and satisfying Assumption \ref{hyp:constant for phieLi}}).
	\end{thrm}

\section{Mathematical framework}

\subsection{Regularity assumptions of the nonlinear terms and initial data}
\label{sec:regularity assumptions}

Our most general formulation in this paper concerns the case of considering the following regularity conditions on the data:
\begin{hyp} \label{hyp:regularity data}
	Let us take the data:
	\begin{alignat*}{2}
	& {c_{\rm e,0}} \in H^1(0,L),  & & c_{e,0} > 0, \\
	& {c_{\rm s,0}} \in \mathcal C  ( {\overline \Done} ) , &  & 0 < c_{s,0} < \csmax \\
	& T_0 > 0 & & \\
	& I \in \mathcal C_{\rm {part}}([0, t_{\textrm {end}}^I]), & \qquad  & 0<{\tend\le} t_{\textrm {end}}^I<+\infty,
	\end{alignat*}
	where {$\mathcal C_{\rm part}$ denotes the set of piecewise continuous functions
	\begin{equation} \label{eq:defn piecewise continuous}
		\mathcal C_{\rm {part}}([a,b])= \{ f : [a,b] \to \mathbb R : \exists \, a = t_0< t_1<\cdots< t_N=b \textrm { such that } f \in \mathcal C([t_{i-1},t_i]) \}. 
	\end{equation} 
	Notice that this implies that the lateral limits $f(t_i^{\pm})$ exist, but need not coincide. }
	\end{hyp}

\begin{hyp} \label{hyp:regularity of coefficients}
	$D_{\rm e} \in L^\infty(0,L)$, $\kappa  \in \mathcal C^2 ((0,+\infty)^2), \sigma \in L^\infty (\J)$, $D_{\rm e} \ge D_{e,0} > 0$, $\kappa \ge \kappa_0 > 0, \sigma \ge \sigma_0 > 0$ and $f_{\phie} \in \mathcal C^2((0,+\infty))$.
\end{hyp}

The key ingredient of our approach is to choose some suitable convex subsets of some appropriate functional spaces. The choice of spaces well adapted to the system is a delicate matter, and errors or loose approaches to this might result in incorrect results (for an explanation of this philosophy see, e.g., \cite{Brezis:1999}). Let us define the open convex sets where we expect to find the solutions:
\begin{align*}
X &= H^1(0,L)\times \mathcal C(\overline{\J}) \times \mathbb R,\\
K_X &=  \Big \{  (\ce, c_{s,B} , T)  \in X : \ce > 0, 0 < c_{s,B} < \csmax , T > 0	\Big \},\\
Y &= L^\infty (0,L) \times \mathcal C (\overline {\J}) \times \mathbb R, \\ 
Z &= H^1(0,L)\times \mathcal C(\overline{\J}) \times H^1(0,L) \times H^1(\J) \times \mathbb R, \\
K_Z &= \Big\{ (\ce,c_{s,B}, \phie, \phis, T) \in Z : (\ce, c_{s,B} , T) \in  K_X\Big\},
\end{align*}
where $H^1(a,b)$ is the usual Sobolev space over the interval $(a,b)$ (see, e.g., {\cite{Brezis:2010,Ramos2012Intro+analisis}}). It is important to point out that $H^1(a,b) \hookrightarrow \mathcal C ([a,b])$. Finally, we define
\begin{equation} \label{eq:defn XPhi}
	{\XPhi} = \left \{ (u,v) \in H^1(0,L) \times H^1(\J) : \int_{0}^L u {(x) \dx} = 0 \right \},
\end{equation}
{the natural space in which we will look for the pair $(\phie, \phis)$.}\\

{Instead of narrowly focusing on \eqref{eq:Butler-Volmer start}-\eqref{eq:Butler-Volmer end} we shall state {an} assumption (sufficient to prove Theorem \ref{thrm:well posedness model}) satisfied by a broader family of functions:}
\begin{hyp} \label{hyp:regularity flux general} For the flux function we assume:
	\begin{align} \label{eq:jLi continuity}
	\overline{j}^\Li &\in \mathcal C^{2} (\J \times (0,+\infty) \times (0,\csmax) \times (0,+\infty) \times \mathbb R),\\
	U &\in \mathcal C^2 (\J \times (0,+\infty) \times (0, \csmax ) \times (0,+\infty)), \label{eq:U continuity}
	\end{align}
	such that
	\begin{equation} \label{eq:monotonicity of jLi with respect to eta} 
			\frac{\partial \overline{j}^\Li}{\partial \eta} (x,  \ce,  \cs,  T,  \eta) > 0, 
	\end{equation}
	{for all $ (x,  \ce,  \cs,  T,  \eta) \in \J \times (0,+\infty) \times (0,\csmax) \times (0,+\infty) \times \mathbb R.$}
\end{hyp}
\begin{rem} \label{rem:monotonicity F Li}
		In particular, it follows from Assumption \ref{hyp:regularity flux general} {(in particular due to \eqref{eq:jLi continuity} and \eqref{eq:monotonicity of jLi with respect to eta}} applying the {Mean} Value Theorem) that there exists a positive continuous function $F^\Li$ satisfying
		\begin{equation} \label{eq:coercivity}
		\Big(\overline{j}^\Li (x,  \ce,  \cs,  T,  \eta ) - \overline{j}^\Li (x,  \ce,  \cs,  T,  {\hat \eta} ) \Big)   ( \eta - {\hat \eta }) = F^\Li \left (x,  \ce,  \cs,  T,  \eta, {\hat \eta} \right ) \left | {\eta} - {\hat \eta} \right |^2,
		\end{equation}
		for all $ x \in \J,  \ce >0 ,  \cs \in (0,\csmax),  T>0$ and  $ \eta , {\hat \eta} \in \mathbb R$. 
\end{rem}

Finally, on the temperature term $F_T$ we will require {the following}:
\begin{hyp} \label{hyp:regularity FT} $
	F_T \in \mathcal C^1(  K_Z  ; \mathbb R) 
$ {(in the sense of {the Fréchet derivative})}.
\end{hyp}

\subsection{Definition of weak solution}

\label{sec:definition of weak solution}

	We introduce the natural space for radial solutions
\begin{equation*}
	H^1_r (0,R) = \{ u :(0,R) \to \mathbb R \textrm{ measurable such that } u(r)r, u'(r)r \in L^2 (0,R) \}
\end{equation*}
with the norm
\begin{equation*}
	\| u \|_{H^1_r(0,R)}^2 = \int_0^R |u(r)|^2r^2 \dr + \int_0^R |u'(r)|^2 r^2 \dr
\end{equation*}
and the space
\begin{equation*}
	L^2 (\J; H^1_r(0,\Rs(\cdot))) =\left \{ u : \Done \to \mathbb R \textrm{ measurable such that }  \int_0^L \|u (x, \cdot)\|_{H^1_r (0,\Rs(x))}^2 \dx < + \infty  \right \}.
\end{equation*}
		{
		\begin{rem}
			Even though we will always present the variables as $(x,t)$, or ${(r,t;x)}$, it will sometimes be mathematically {advantageous} to consider maps $t \in [0,t_0] \mapsto u(\cdot,t) \in X$, where $X$ will be a space of spatial functions, in which we will use the notations $\mathcal C([0,t_0]; X)$ or $L^2(0,t_0;X)$, depending on the regularity .
		\end{rem}
		}
\begin{defn} \label{defn:weak solution}
	We define a weak solution of \eqref{eq:model ce} - \eqref{eq:model T} as a quintuplet
	\begin{align*}
	(\ce, \cs, \phie, \phis , T) & \in L^2(0,\tend; \widetilde K_Z), \\
	\widetilde K_Z &= H^1(0,L) \times L^2 (\J;  {H^1_r(0,{\Rs (\cdot)})}) \times H^1(0,L) \times H^1(\J) \times \mathbb R,
	\end{align*}
	 such that
			\begin{align} 
			- \int \limits_0^\tend  \int \limits_{0}^L  \ce {\frac {{\rm d}\eta_e}{\dt}} \phietest \dx \dt
			 & +\int \limits_0^\tend  \int \limits_{0}^L  D_e \frac{\partial \ce }{\partial x} {\frac{{\rm d} \phietest }{\dx}} \eta_e \dx \dt
			  = \int \limits_0^\tend  \int \limits_{0}^L {\alpha_{\rm e}} j^\Li \eta_e \phietest \dx -   \int \limits_{0}^L  c_{e,0} \eta(0) \phietest \dx ,
			\end{align}
			for all $\eta_e \in \mathcal D([0,\tend)) = \{ f \in \mathcal C^\infty([0,\tend)) : \textrm{supp} f \subset [0,\tend) \textrm{ is compact} \} $ {(see, e.g. \cite{Lions:1969,Ramos2012Intro+analisis})} and $\phietest \in H^1 (0,L)$; {For a.e. $x \in \J$}
			\begin{align}
			-  \int \limits_0^\tend   \int \limits_{0}^{\Rs(x)}  \cs {\frac{{\rm d}\eta_s}{\dt}} {\cstest} r^2 \dr \dt &+  \int \limits_0^\tend   \int \limits_{0}^{\Rs(x)}  D_s  \eta _s  \frac{\partial \cs}{\partial r}  \frac{\partial {\cstest }}{\partial r} r^2 \dr  \dt \nonumber \\
			& = {-}\int \limits_0^\tend {\alphacseq(x)}  j^\Li \eta_s \cstest  \dt -    \int \limits_{0}^{\Rs(x)}  c_{s,0} \eta(0) {\cstest} r^2\dr  ,
			\end{align}
			for all $\eta_s \in \mathcal D([0,\tend)) $ and  $\cstest \in {H^1_r (0,\Rs(x))}$ radially symmetric. For every $t \in (0,\tend)$
			\begin{align}
			\label{eq:weak formulation phie}
			\int_0^L  \kappa \partiald {\phie} x {\frac{{\rm d}\phietest}{\dx}} \dx - \int_{\J} j^{\Li} \phietest \dx &=  \int_0^L  \kappa \alpha_{\phie} T  \partiald{}{x} (f_{\phie} (\ce )) {\frac{{\rm d}\phietest}{\dx}} \dx, \qquad \forall\, \phietest \in H^1 (0,L)  \\
			\label{eq:weak formulation phis}
			\int_{\J}  \sigma \partiald{\phis}{x} {\frac{{\rm d} \phistest }{\dx}} \dx +  \int_{\J} j^{\Li} \phistest \dx &= { \frac{I(t)}{A}} (\phistest(0) - \phistest(L)), \qquad \forall\, \phistest \in H^1 (\J)
			\end{align}
			and
			\begin{align}
			T(t) &= T_0 + \int_0^t ( - \alpha_T (T(s) - \Tamb)) + \nonumber \\
			& \qquad \int_0^t F_T(\ce {(\cdot, s)}, \csB {(\cdot, s)} , \phie{(\cdot, s)}, \phis{(\cdot, s)}, T(s) ))\ds. 
			\end{align}
\end{defn}

\subsection{Definition of weak-mild solution}

Dealing directly with the weak formulation is technically very difficult. On the other hand, solutions in the classical sense (having all the necessary derivatives) may not exist. There is an intermediate type of solutions, known as ``mild solutions''. As a general rule (and in particular this applies to our problem), any classical solution is a mild solution and any mild solution is a weak solution.\\

Let us introduce this kind of solutions in the simplest case: the heat equation. Consider the problem
\begin{equation} \label{eq:inocent heat}
	\begin{dcases} 
		\frac{\partial u}{\partial t} - \Delta u = f  & \textrm{in }  \Omega {\times (0,\tend)}, \\
		u = 0 & \textrm {on } \partial \Omega {\times (0,\tend)},  \\
		u(\cdot, 0) = u_0, & \textrm{in } \Omega,
		\end{dcases}
\end{equation}
in a bounded, smooth domain $\Omega \subset \mathbb R^N$. Let $u_0 \in L^2 (\Omega)$ and $f \in L^2 ((0, {\tend}) \times \Omega)$. One can construct, as an intermediate step, the solution of the following problem
\begin{equation} \label{eq:inocent heat v}
	\begin{dcases} 
		\frac{\partial v}{\partial t} - \Delta v  = 0  & \textrm{in }  \Omega{\times (0,\tend)}, \\
		v = 0 & \textrm {on } \partial \Omega {\times (0,\tend)},  \\
		v(\cdot,0) = u_0, & \textrm{in } \Omega,
	\end{dcases}
\end{equation}
by considering the decomposition of $L^2 (\Omega)$ in terms of eigenfunctions of $-\Delta$. Let us write the unique solution of \eqref{eq:inocent heat v} as $v (t) = S(t) u_0$. The operator $S(t)$ is a semigroup (see \cite{Brezis:2010}), and has some interesting properties we will not discuss. A solution $u$ of problem the non homogeneous problem \eqref{eq:inocent heat} can be written, for every $t \in [0,\tend]$, as
\begin{equation} \label{eq:inocent variation of constants formula}
	u(t) = S(t) u_0 + \int_0^t S(t-s ) f(s) \ds .
\end{equation}
This kind of solution is known as a ``mild solution''. As in \cite{diaz+vrabie1989compacite+green}, one can define the ``Green operator'' for problem {\eqref{eq:inocent heat}} as the function
\begin{equation}
	G_{\tend} : f \mapsto S(\cdot) u_0 + \int_0^\cdot S(\cdot -s ) f(s) \ds {.}
\end{equation}

In our problem we will need to work with a suitable Green operator associated to each of the equations. {Assuming} Assumptions \ref{hyp:regularity of coefficients}, {we will define several Green operators.\\}

{For any $t_0 > 0$ we define} (see \cite{friedman2008partial}):
		\begin{eqnarray*} 
			G_{\ce, t_0}  :  L^2( (0,L){\times (0,t_0)}) &\to& \mathcal C([0,t_0]{; } H^1(0,L)) , \\
			  f &\mapsto&  V ,
		\end{eqnarray*}
		as the solution of the problem
		\begin{equation*}
			\begin{dcases}
			\partiald {V} t - \frac{\partial}{\partial x} \left( D_e \frac{\partial }{\partial x }  V \right) = 	
			f , & (x,t) \in (0,L) \times (0,t_0),\\
			\partiald {V} {x} (0,t) = \partiald {V} {x} (L,t) = 0, & t\in (0,t_0),\\
			V(x,0) = c_{e,0}(x), & x \in (0,L).
			\end{dcases}
		\end{equation*}

		For {system \eqref{eq:model cs}} we will need to do some extra work due to the fact that the equation is only ``pseudo 2D''. First we define the solution {of problem \eqref{eq:model cs}} for every $x$ fixed
		\begin{eqnarray*}
			\label{defn G cs R x} 
			G_{\cs, R, t_0}  : \mathcal C([0,R]) \times \mathcal C ([0,t_0]) &\to& \mathcal C([0,R]{\times[0,t_0]}), \\
			 (u_0,g) &\mapsto& V, 
		\end{eqnarray*}
		by solving the corresponding problem
		\begin{equation*}
			\begin{dcases}
			\partiald {V} t -  \frac{1}{r^2} \frac{\partial }{\partial r} \left( D_s r^2 V \right) = 	
			0, & (y,t) \in (0,R) \times (0,t_0),\\
			-D_s \partiald {V} {r} (R,t) = g, & t\in (0,t_0),\\
			V(r,0) = u_0(r), & r \in (0,R).
			\end{dcases} 
		\end{equation*}
		The next step is to consider the dependence on $x$. Therefore we construct the {Green operator} {associated to problem \eqref{eq:model cs} collecting all $x \in \J$}:
			\begin{eqnarray*}
			G_{\cs, t_0}  : \mathcal C({ \J \times [0,t_0] }) &\to& \mathcal C({\Done} \times [0,t_0]  ) , \\
			 g &\mapsto& W,
		\end{eqnarray*}
		given by
		\begin{equation*}
			W(r,x{,t}) = G_{\cs, \Rs(x),t_0} (c_{s,0} ({x,\cdot}),g({x,\cdot})) (r, {t}).
		\end{equation*}
		Finally, we consider the Green operator for the  {system \eqref{eq:model T}} as the function
		\begin{eqnarray*}
			G_{T,t_0} : \mathcal C ([0,t_0];Z ) &\to& \mathcal C([0,t_0]) , \\
			 (\ce, \csB, \phie, \phis, T) &\mapsto&  W 
		\end{eqnarray*}
		defined as
		\begin{equation*}
			W(t) =  T_0 + \int_0^t ( - \alpha_T (T(s) - \Tamb) + 
			 \int_0^t  F_T(\ce {(\cdot, s)}, \csB {(\cdot, s)} , \phie{(\cdot, s)}, \phis{(\cdot, s)}, T(s) ))\ds. 
		\end{equation*}
		{{This operator is well-defined and of class $\mathcal C^1$} due to Assumption \ref{hyp:regularity FT}.}
	It will be useful to introduce the following {Nemistkij} operators:
	\begin{eqnarray*}
	N_{j^\Li}  :  K_Z &\to&  \mathcal C ([0,L_1]) \cap \mathcal C ([L_1, L_1 + \delta]) \cap \mathcal C([L_1+\delta, L]) \\
				(\ce, c_{s,B}, \phie, \phis, T) & \mapsto & j^\Li \circ (\ce, c_{s,B}, \phie, \phis, T) 	\\ \nonumber
		\\		
		N_{j^\Li,t_0}  : \mathcal C ([0,t_0]; K_Z) &\to& \mathcal C ([0,t_0]; \mathcal C ([0,L_1]) \cap \mathcal C ([L_1, L_1 + \delta]) \cap \mathcal C([L_1+\delta, L])) \\
		(\ce, c_{s,B}, \phie, \phis, T) & \mapsto & j^\Li \circ (\ce, c_{s,B}, \phie, \phis, T). 					
	\end{eqnarray*}
	It is well known (see, e.g., \cite{henry1981geometric+theory+semilinear+parabolic}) that these operators are locally Lipschitz continuous and $\mathcal C^1$ {(in the sense of {the Fréchet derivative})}, properties that will be used in the proof {of Theorem \ref{thrm:well posedness model}}{,} due to the regularity of the elements of the composition {(i.e. \eqref{eq:jLi continuity} and \eqref{eq:U continuity})}.

	\begin{defn}[Weak-mild solution] \label{defn:solution general problem}
		We define a ``weak-mild solution of \eqref{eq:model ce}--\eqref{eq:model T}'' as {a quintuplet} {$(\ce, \cs, \phie, \phis, T)\in C([0,\tend); K_Z$)} such that there exists $0 < \tend \le \tendI$ for which:
		\begin{enumerate}
			\item {$(\phie,\phis)$ is a} weak solutions of {system \eqref{eq:model phie}--\eqref{eq:model phis}} for {the functions} $\ce, \cs, T$ given {in the quintuplet}, in the sense that, for every $t \in [0, \tend)$ the weak formulations \eqref{eq:weak formulation phie}, \eqref{eq:weak formulation phis} hold.
			
			\item $(\ce, \cs, T)$ {is a mild solutions of the system \eqref{eq:model ce}, \eqref{eq:model cs}, \eqref{eq:model T}} for {the functions} $\phie$, $\phis$ given {in the quintuplet}, in the sense that for every $t_0 < \tend$:
			\begin{align} \label{eq:solution as fixed point in cs}
			(\ce, \cs, T) =& \left(G_{\ce,t_0} \left ({\alphaceeq} N_{j^\Li,t_0}  \right ) , G_{\cs , t_0}  \left ({\alphacseq} N_{j^\Li,t_0}  \right), G_{T,t_0}  \right) \nonumber \\
			&\quad \circ (\ce|_{t<t_0}, \cs |_{R = \Rs(x), t<t_0}, \phie|_{t<t_0} , \phis|_{t<t_0},  T|_{t<t_0}).
			\end{align} 
		\end{enumerate}
	\end{defn}
	\begin{defn}[{Piecewise} weak-mild solution] \label{defn:solution general problem by parts} We define a ``{piecewise} weak-mild solution'' as a quintuplet $(\ce, \cs, \phie, \phis, T)$ such that there exists a partition $\{t_0, \cdots, t_N\}$ of $[0,\tend]$ such that in $[t_i, t_{i+1}]$ $(\ce, \cs, \phie, \phis, T)$ is a weak-mild solution in the previous sense, with $(\ce, \cs, T) (t_i)$ as initial condition in the interval $[t_i, t_{i+1}]$. 
	\end{defn}
	 
	\begin{rem}
		It is well known that {for problems of type \eqref{eq:inocent heat},} any {piecewise} {weak-mild} solution is a weak solution.
	\end{rem}
	
	\begin{defn}\label{defn:extension}
		Given a solution {(in any {of the previous senses})} $(\ce, \cs, \phie, \phis, T) \in \mathcal C([0,a) , K_Z)$, we say that $(\tilde \ce, \tilde\cs, \tilde\phie, \tilde\phis,\tilde T) \in \mathcal C([0,b) , K_Z)$ is an ``extension'' of $(\ce, \cs, \phie, \phis, T)$ if {it is also a solution (in the same sense)}, $b \ge a$ and 
		\begin{equation*}
			(\tilde \ce, \tilde\cs, \tilde\phie, \tilde\phis,\tilde T)|_{t{\le}a} = (\ce, \cs, \phie, \phis, T).
		\end{equation*}
		We say that the extension is ``proper'' if $b > a$.
		We say that an extension is ``maximal'' if it {does not} admit {a} proper extension.
	\end{defn}

	Notice that the contribution of $\cs$ can be studied, basically, as a 1D behaviour on $\J$. {More precisely, we} consider the Green operator {on the boundary of the balls $B_{\Rs(x)}$}, which we will show that contains all the necessary information
	\begin{eqnarray*}
		G_{\csB, t_0} : \mathcal C( {\overline \J \times [0,t_0]}) &\to& \mathcal C ({ \overline \J \times [0,t_0]}), \\
		g &\mapsto& W ,
	\end{eqnarray*}
	where
	\begin{equation*}
		 W({x,t}) = (G_{\cs, t_0} (g))({\Rs(x),x,t}).
	\end{equation*}
	In this sense we can rewrite \eqref{eq:solution as fixed point in cs} in terms of the restriction $\csB$, instead of $\cs$, as follows:

	\begin{prop} {{In} Definitions \ref{defn:solution general problem by parts} and \ref{defn:extension},} condition \eqref{eq:solution as fixed point in cs} is equivalent to the following property: $(\ce, \csB, T) \in \mathcal C([0,\tend) ; X)$ such that
		\begin{align}  \label{eq:solution as fixed point in csB}
			(\ce, c_{s,B}, T) =& \left(G_{\ce,t_0} \left ({\alphaceeq} N_{j^\Li,t_0}  \right ) , G_{\csB , t_0}  \left ({\alphacseq} N_{j^\Li,t_0}  \right), G_{T,t_0}  \right) \nonumber\\
			& \quad \circ (\ce|_{t<t_0}, \csB |_{t<t_0}, \phie|_{t<t_0} , \phis|_{t<t_0},  T|_{t<t_0}){.}
		\end{align} 
	\end{prop}
	\begin{proof}
		{It is trivial that} \eqref{eq:solution as fixed point in cs} implies \eqref{eq:solution as fixed point in csB}. Let us consider a quintuplet $(\ce, c_{s,B}, \phie, \phis, T)$ that satisfies \eqref{eq:solution as fixed point in csB}. Let us construct a solution $(\ce, \cs, \phie, \phis, T)$ of \eqref{eq:model ce}--\eqref{eq:model T} in the sense of Definition \ref{defn:solution general problem}. Define, for any $t_0 < \tend$
		\begin{equation} \label{eq:coupling nemitskij reconstruction}
		\widetilde {j^\Li} = N_{j^\Li,t_0}  (\ce|_{t<t_0}, \csB|_{t<t_0}, \phie|_{t<t_0}, \phis|_{t<t_0}, T|_{t<t_0}).
		\end{equation}
		Define $\cs (x, r,t )$ to be the solution, for every $x \in \J$, of the problem
		\begin{equation} \label{eq:parabolic Dirichlet problem cs}
			\begin{dcases}
				\partiald{\cs}{t} - D_s  \frac{1}{r^2} \frac{\partial}{\partial r} \left( r^2 \frac{\partial \cs}{\partial r}\right) = 0 ,  & 0 < r < \Rs (x), \\
				\cs = \csB , & r = \Rs (x), \\
				\frac{\partial \cs}{\partial r} = 0, & r = 0, \\
				\cs = c_{\textrm{s},0}, & t = 0.
			\end{dcases}
		\end{equation}
		Since $\cs = \csB$ on $\partial B_{\Rs (x)}$, due to \eqref{eq:coupling nemitskij reconstruction},  we have that
		\begin{equation*}
				\widetilde {j^\Li} (x,t) = j^{\Li} (x,\ce(x,t), \cs (\Rs(x),x,t), \phie(x,t), \phis(x,t), T(t)), \qquad (x,t) \in \J \times (0,t_0) .
		\end{equation*}
		Furthermore, since $\csB = (G_{\cs, t_0} \widetilde {j^\Li})|_{R = \Rs (x)}$ then $\cs$ and $G_{\cs,t_0} \widetilde{j^\Li }$ are solutions of the same parabolic problem \eqref{eq:parabolic Dirichlet problem cs} {and,} by the uniqueness of the solutions,		\begin{equation*}
			\cs = G_{\cs,t_0} \widetilde{j^\Li }.
		\end{equation*}
		Therefore $(\ce, \cs, \phie, \phis, T)$ is a solution of \eqref{eq:model ce}--\eqref{eq:model T} in the sense of Definition \ref{defn:solution general problem}, up to time $t_0${, which can be} arbitrarily close to $\tend$.
	\end{proof}

\subsection{{Assumptions and results regarding Theorem \ref{thrm:well posedness model}}} 

It was first shown in \cite{Ramos:2015:lithiumionbatteries} that the uniqueness of solutions $(\phie,\phis)$ {of system \eqref{eq:model phie}--\eqref{eq:model phis}} holds up to a constant {relating the difference between $\phie$ and $\phis$}. To avoid this we set the following assumption:
\begin{hyp} \label{hyp:constant for phieLi} As in \eqref{eq:defn phieLi}, we define
	\begin{align} \label{eq:definition phieLi}
	\phieLi = \phie - \alpha_{\phie} T f_{\phie} (\ce) ,
	\end{align} 
	and assume that
	\begin{align}
	&\int_{0}^{L} \phieLi \dx =  0. \label{eq:phieLi 0 average}
	\end{align}
	{This can be done because $(\phie, \phis)$ is defined up to a constant}. 
\end{hyp}

{
	\begin{rem}
		We recall that $\mathcal C([a,b]) \subset H^1 (a,b)$. Thus, since $0 < \ce \in H^1 (a,b)$ then, $\min_{[0,L]} \ce > 0$, so $f_{\phie} (\ce) \in H^1 (0,L)$.
	\end{rem}
}

	\begin{rem}
		{Another alternative {to get the uniqueness of solution} is to use} the condition $\phis|_{x = 0} = 0$, {instead of \eqref{eq:phieLi 0 average},} setting the value $0$ of the potential {in one of the walls}.
	\end{rem}

The idea of {the} proof {of Theorem \ref{thrm:well posedness model}} {(which is done in Section \ref{sec:proof of theorem 1})} is the following. First we will show {(see Proposition \ref{prop:resolvent elliptic} {and its proof in Section \ref{sec:existence of Green operator}})} that we can solve \eqref{eq:model phie} and \eqref{eq:model phis} to obtain $(\phie, \phis)$ if $\cs, \ce, T$ {are given}, {extending to the nonlinear case the study for the linearized equation proved in \cite{Ramos:2015:lithiumionbatteries}}. Then we will apply a fixed point argument to the evolution problems \eqref{eq:model ce}, \eqref{eq:model cs} and \eqref{eq:model T} to obtain the conclusion.

\begin{prop} \label{prop:resolvent elliptic}
	Let $(\ce, \csB, T) \in K_X$, {$I\in \mathbb R$}  and let Assumptions \ref{hyp:regularity of coefficients}, \ref{hyp:regularity flux general} hold. Then there exists $\phie \in H^1(0,L)$ and $\phi_{s} \in H^1 (\J)$ satisfying the elliptic equations \eqref{eq:model phie} and \eqref{eq:model phis} in the weak sense \eqref{eq:weak formulation phie} and  \eqref{eq:weak formulation phis}.  Furthermore, given two solutions $(\phie, \phis), (\hat \phie, \hat \phis)$ there exists a constant $C \in \mathbb R$ {such that}
	$$\phie - \widehat \phie = \phis - \widehat \phis = {C}. $$
	Hence we have uniqueness up to a constant. In particular, there exists a unique solution $(\phie, \phis)$ satisfying Assumption \ref{hyp:constant for phieLi}.
\end{prop}

Due to this proposition we {k}now that {the following functions are well defined:}
\begin{eqnarray}
	G_{\phi}: \quad K_X \times {\mathbb R} &\to& H^1(0,L) \times H^1(\J) \nonumber \\
		(\ce, \csB, T, I) &\mapsto& (\phie, \phis), \label{eq:defn Gphi}\\
	\widetilde G_{\phi}: \quad K_X \times {\mathbb R} &\to& K_Z \nonumber \\
	(\ce, \csB, T, I) &\mapsto& (\ce, \csB, G_\phi(\ce, \csB, T, I), T),
\end{eqnarray}
where $(\phie, \phis) \in {H^1(0,L) \times H^1(\J)}$ {is} the {(unique)} solution of \eqref{eq:weak formulation phie}--{\eqref{eq:weak formulation phis}} {satisfying} \eqref{eq:phieLi 0 average}. Assuming {now that} $I$ is a continuous function{, i.e.} $I \in \mathcal C([0,\tendI])$, we define the Green {operator}, for $t_0 < \tendI$: 
\begin{eqnarray*}
	{\widetilde  G_{\phi,t_0} }: \mathcal C([0,{t_0}];K_X) &\to& \mathcal C ([0,{t_0}];K_Z)\\
		(\ce, \csB, T) &\mapsto& W, 
\end{eqnarray*}
where
\begin{equation*}
	 W(t) = { \widetilde G_{\phi} } (\ce(t), \csB(t), T(t), I(t)).
\end{equation*}
In Section \ref{sec:existence of Green operator} we will prove that
\begin{prop}\label{prop:resolvent elliptic Lipschitz}
	Let Assumptions \ref{hyp:regularity of coefficients}, \ref{hyp:regularity flux general} hold. Then, the operator {$\widetilde G_\phi : K_X \times \mathbb R \to K_Z$} is $C^1$ {(in the sense of {the Fréchet derivative})}.
\end{prop}
{
	\begin{rem}
		Since we will allow for charge and discharge cycles, we allow for $I$ to be piecewise continuous, and this is why we define the {piecewise} {weak-}mild solution {(see Definition \ref{defn:solution general problem by parts})}.
	\end{rem}
}

The proof of the local existence of solutions will be based on finding a unique fixed point{, in $\mathcal C([0,t_0];K_X)$ for $t_0$ small enough, of} the operator problem
\begin{align*}		
	(\ce, \csB, T) =& \left(G_{\ce,t_0} \left (\alphaceLi N_{j^\Li,t_0}  \right ) , G_{{\csB} , t_0}  \left (\alphacsLi N_{j^\Li,t_0}  \right), G_{T,t_0}  \right) \circ  { \widetilde G_{\phi,t_0} } (\ce, \csB ,  T).
	\end{align*}

\subsection{{Assumptions and remarks regarding Theorem \ref{thrm:blow up behaviour}}}

In Theorem \ref{thrm:well posedness model} one of the reasons of a finite existence time could be that $\csB \to 0, \csmax$ or $\ce \to 0$. This conditions do not, a priori, pose a relevant physical problem since the battery may very well be completely full or empty. However, the generality of our setting allows for no better statement. Let us study the case which seems to be the most relevant for the modeling of Lithium-ion batteries {by considering new assumptions}.\\

\begin{hyp} \label{hyp:constant coefficients}
	There exists a constant $\kappa_1$ such that $\kappa (\ce, T) \le \kappa_1$ and $f_{\phie} (\cdot) = \ln (\cdot) $.
\end{hyp}

\begin{hyp} \label{hyp:csmax less 1}
	$
	\csmax  < 1
	$ {in the units considered to solve the problem}.
\end{hyp} 
This is purely technical, but it seems reasonable since, in empirical cases in the literature, typically $\csmax \sim 10^{-2} \textrm{ mol } \textrm{cm}^{-3}$. In particular, in \cite{smith+wang2006diffusion+limitation+lithium} the authors take $\csmax = 1.6 \times 10^{-2} {\textrm{ mol } \textrm{cm}^{-3}}$.

Let us consider the following special nonlinear terms {(which, for a broader generality, include some extra constants with respect to \eqref{eq:Butler-Volmer start}--\eqref{eq:Butler-Volmer end})}, {which we will take as assumption in Theorem \ref{thrm:blow up behaviour}}

	\begin{hyp} \label{hyp:flux ramos please}
		We assume that
	{
	\begin{align*}
	\bar j^{\Li} {(x, \ce, \cs, T, \eta)} &= \ce^{\alphaceLi} \cs^{\alphacsLi} (\csmax - \cs)^{\betacs} h \left ({x,} \frac{\eta}{T}  \right ) ,\\
	h{(x,s)} &= h_+ {(x,s)} - h_- {(x,s)}, \\
	h_+{(x,s)} &= { \delta_1 (x) }\exp ( \gamma_1 s) , \\
	h_- {(x,s)} &= { \delta_2 (x) } \exp ({-}\gamma_2 s), \\
	\eta(x, \ce, \cs, \phie, \phis, T) &=  \phis - \phie - U(x,\ce, \cs,T) . 
	\end{align*}
	}{where $\alphaceLi,\alphacsLi,\betacs \in (0,1)$, $\gamma_1, \gamma_2 > 0$} {and {$\delta_1(x), \delta_2(x) > 0$ are} constant in each electrode.} {Notice that $\gamma_1, \gamma_2$ are constant but not dimensionless.}
	{Furthermore, we} consider $U$ slightly more general than \eqref{eq:open circuit potential}:
	\begin{align} \label{eq:open circuit potential general}
	U( {x}, \ce ,  \cs,  T) &= - \lambda_{\min}({x}, T) \ln  \cs + \lambda_{\max}({x}, T) \ln (\csmax -  \cs) + \mu({x}, T) \ln \ce + p( {\ce, \cs}, T), 
	\end{align}
	where $\lambda_{\min}, \lambda_{\max}, \mu$ are given smooth nonnegative scalar functions and $p${ is a continuous bounded function } {$[0,+\infty) \times [0, \csmax] \times (0,+\infty) \to \mathbb R$ with} global bound {denoted} {by} $\| p \|_{L^\infty}$ {(Figure \ref{fig:U} shows a typical graph of $U$)}.
	\end{hyp}
{
	\begin{rem}
	It is a routine matter to check that this $\overline{j}^\Li$ satisfies Assumption \ref{hyp:regularity flux general}.
	\end{rem}
}

\begin{figure}[H]
	\centering
	\includegraphics[]{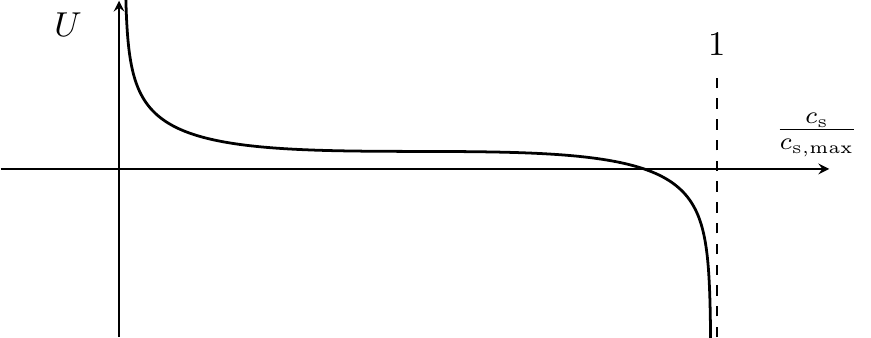}
	\caption{Possible representation of $U$}
	\label{fig:U}
\end{figure}

\begin{rem} 
	Under Assumption \ref{hyp:flux ramos please} we can write $\overline{j}^\Li {(x, \ce, \cs, T, \eta(x, \ce, \cs, \phie, \phis, T))} =  \overline{j}^\Li_+ - \overline{j}^\Li_-$, {where}
	\begin{align}
	\bar j^{\Li}_+ & =  {\delta_1 (x)} \ce^{\alphaceLi} \cs^{\alphacsLi} (\csmax - \cs)^{\betacs} \exp \left ( \frac {\gamma_1} T \left( \phis - \phie - U({x},{\ce},\cs,T) \right) \right ), \label{eq:jLi+ intermediate}\\
	\intertext{and}
	\bar j^{\Li}_- & =   {\delta_2 (x)} \ce^{\alphaceLi} \cs^{\alphacsLi} (\csmax - \cs)^{\betacs} \exp \left ( -\frac {\gamma_2} T \left( \phis - \phie - U({x}, \ce, \cs,T) \right) \right ) \label{eq:jLi- intermediate}. 
	\end{align}	
	Substituting \eqref{eq:open circuit potential general} into \eqref{eq:jLi+ intermediate}-\eqref{eq:jLi- intermediate} we have 
	\begin{align} 
		\overline{j}^\Li_+ & =  { \delta_1 (x)} \ce^{\alphaceLi - \gamma_1 \left( \alpha_{\phie} + \frac{\mu ({x},T)}{ T} \right)} \cs^{\alphacsLi  + \gamma_1 \frac{\lambda_{\min} ({x},T)}{T}} (\csmax - \cs)^{\betacs - \gamma_1 \frac{\lambda_{\max}({x},T)}{T}} \nonumber \\ 
		&\qquad \times \exp \left (\frac {\gamma_1} T (\phis - \phieLi ) \right ) \exp \left( \frac{{-}\gamma_1}{T}p (\ce, \cs, T)\right), \label{eq:fLi plus}\\
		\overline{j}^\Li_- & = { \delta_1(x)}\ce^{\alphaceLi + \gamma_2 \left( \alpha_{\phie} + \frac{\mu ({x},T)}{ T} \right)} \cs^{\alphacsLi  - \gamma_2 \frac{\lambda_{\min} ({x},T)}{T}} (\csmax - \cs)^{\betacs + \gamma_2 \frac{\lambda_{\max}({x},T)}{T}} \nonumber \\ 
		&\qquad \times \exp \left (\frac {-\gamma_2} T (\phis - \phieLi ) \right ) \exp \left( \frac{\gamma_2}{T}p (\ce, \cs, T)\right),  \label{eq:fLi minus}
	\end{align}
	\end{rem}

{On the exponents we will consider the following assumption {(see Remark \ref{rem:dead core})}:}

\begin{hyp} \label{hyp:bounds exponents}
	For all $ T > 0$ {and $x\in [0,L]$}
	\begin{align} 
	\label{eq:blow up potentials first condition}
	\alphaceLi - \gamma_1 \left(  \alpha_{\phie}  + \frac{\mu ({x},T)}{ T}  \right) & \le 1 \\
	\alphaceLi + \gamma_2 \left ( \alpha_{\phie}  + \frac{\mu ({x},T)}{ T} \right)  & \ge 1  \\
	\alphacsLi + \gamma_1 \frac{\lambda_{\min} ({x},T)}{  T} & \ge  1, \qquad  \\
	\label{eq:blow up potentials last condition}
	\betacs + \gamma_2 \frac{\lambda_{\max} ({x},T)}{  T} & \ge 1.
	\end{align}
\end{hyp}

\begin{rem} 
	Notice that the conditions \eqref{eq:blow up potentials first condition}-\eqref{eq:blow up potentials last condition} can be {also} written as:
	\begin{gather}
		\lambda_{\min} ({x},T) \ge  \frac{1 - \alphacsLi }{\gamma_1}  T > 0, \qquad \lambda_{\max}  (x, T ) \ge \frac{1 - \betacs}{\gamma_2}  T > 0, \label{eq:no blow up condition exponents 1}\\ 
		{\alpha_{\phie} + \frac{\mu ({x},T)}{ T}  \ge  \frac{1 - \alphaceLi }{\gamma_2} , \qquad \alpha_{\phie} + \frac{\mu ({x},T)}{ T}  \ge \frac{\alphaceLi - 1}{\gamma_1} }\label{eq:no blow up condition exponents 2}.
	\end{gather}
	{ Theorem \ref{thrm:blow up behaviour} gives a sufficient condition so that, under the structural consideration of a potential $U$ similar to the one considered by Ramos-Please \cite{ramos+please2015arxiv}, should there be a blow-up of the solution, this is not caused by the non-physical behaviour $\ce \to 0, +\infty$ or $\cs \to 0, \csmax$ as $t \to \tend$. Remark \ref{rem:dead core} gives an intuition of why this happens. On the other hand, if this conditions are not satisfied, then the toy models \eqref{eq:inocent semilinear root}-\eqref{eq:inocent semilinear power} suggest that if the sufficient conditions in Theorem \ref{thrm:blow up behaviour} are not satisfied, then it could happen that $\ce \to 0, +\infty$ or $\cs \to 0, \csmax$. The suggestion of the potential in Ramos-Please \cite{ramos+please2015arxiv} was motivated by physical consideration, and not by the mathematical theory. Nonetheless, this seems to be the exact right structure for the mathematical theory. This is a notion of robustness to this proposal. } 
	
	{Considering the parameters used by some authors (see, e.g., \cite{smith+wang2006diffusion+limitation+lithium})}
	\begin{gather}
	\alphacsLi = \betacs = \alphaceLi = \frac 1 2 \\
		\gamma_1 = \frac{\hat \alpha_a F}{R} = 5805.5 \ \rm { C\, K \,  {\rm J^{-1}}} \qquad \qquad 
		\gamma_2 = \frac{\hat \alpha_c F}{R} = 5805.5 \ \rm { C\, K \,  {\rm J^{-1}}}
	\end{gather}
	(where $\hat \alpha_a,  \hat \alpha_c$ are non-dimensional charge transfer coefficients, $F$ is {the} Faraday's constant and $R$ is the universal gas constant), {we have
	\begin{equation}
		\frac{1 - \alphacsLi }{\gamma_1} = \frac{1 - \betacs}{\gamma_2} = \frac{1 - \alphacsLi}{\gamma_1} = \frac{1 - \alphacsLi}{\gamma_2} =  8,61 \times 10^{-5} { \rm  { J\ C}^{-1} \, \rm{K}^{-1} }.
	\end{equation}
	}{Considering \eqref{eq:open circuit potential}, as in \cite{ramos+please2015arxiv}, we have 
	\begin{equation}
		\lambda_{\min} = \alpha(x) T, \quad \lambda_{\max}  = \beta(x) T , \quad  \mu (x) = \gamma (x) T.
	\end{equation}
	And so Assumption \ref{hyp:bounds exponents} translates into
	\begin{equation}
		\alpha(x) , \beta(x) \ge 8.61 \times 10^{-5} {\mbox{JC}^{-1} \mbox{K}^{-1}}, \qquad {\gamma}(x) \ge -\alpha_{\phie} + 8.61 \times 10^{-5}  {\mbox{JC}^{-1} \mbox{K}^{-1}}.
	\end{equation}
}
\end{rem} 
\begin{rem} \label{rem:dead core}
	It is well known that equations of the form 
	\begin{equation} \label{eq:inocent semilinear root}
	u_t - \Delta u + u^q = f \ge 0
	\end{equation}
	with $q < 1$ can have a solution $u \not \equiv 0$ such that $\{x \in (0,L) : u(x,t) = 0\}$ has positive measure for all {time $t$ after an initial time $t_0$}. These regions are known as ``dead cores'' and a detailed analysis of this phenomenon can be found, for instance, in \cite{Diaz:1985,diaz2001qualitative+nonlinear+parabolic}. 	{On the other hand, if we study the equation 
	\begin{equation} \label{eq:inocent semilinear power}
		u_t - \Delta u = u^q + f,
	\end{equation}
	we see that, if $q >1$ then the solution can blow up $u \to +\infty$ for a finite time, whereas there is no blow up if $q \le 1$.}
	As we will see in the proof of Theorem \ref{thrm:blow up behaviour}, {we} can bound $\ce$ and $\cs$ by solutions of problems {of type \eqref{eq:inocent semilinear root}}, where the  
	conditions {\eqref{eq:blow up potentials first condition}--\eqref{eq:blow up potentials last condition}} define the exponents $q$ {appearing in \eqref{eq:inocent semilinear root} {and \eqref{eq:inocent semilinear power}}}. {Taking into account these results it seems that}, if {these} conditions are not satisfied, it is likely that the solutions $\ce$ or $\cs$ reach $0$ in finite time, {which} is a non-physical behaviour.

\end{rem}

\subsection{{Assumptions and remarks regarding Theorems \ref{prop:blow up truncated potential} and \ref{prop:global existence truncated}}} \label{sec:assumptions theorem 3 and 4}

Tackling an example of global existence of solutions for the model {seems to be} a {Herculean} mission. However, if we assume some nicer behaviour of the nonlinear term $U$, at least far from the natural ``working'' conditions of the system, then we can establish a result {of} global existence of solutions.

\subsubsection{Truncated potential behaviour}

In this section we strive to make some small modifications to the problem so that we can show that the charge potentials do not blow up in finite time. First let us assume that for large difference of potentials the flux is constant. We slightly modify \eqref{eq:fLi plus}-\eqref{eq:fLi minus} to the following new hypothesis:
\begin{hyp} \label{hyp:flux truncated}
	Let $\overline{j}^\Li {(x, \ce, \cs, T, \eta(x, \ce, \cs, \phie, \phis, T))} = \overline{j}^\Li_+ - \overline{j}^\Li_-$ 
	{ where
		\begin{align} 
		\overline{j}^\Li_+ & = \ce^{\alphaceLi - \gamma_1 \left( \alpha_{\phie} + \frac{\mu ({x},T)}{ T} \right)} \cs^{\alphacsLi  + \gamma_1 \frac{\lambda_{\min} ({x},T)}{T}} (\csmax - \cs)^{\betacs - \gamma_1 \frac{\lambda_{\max}({x},T)}{T}} \nonumber  \\
		& \qquad \times  H \left (\frac {\gamma_1} T (\phis - \phieLi ) \right ) \exp \left( \frac{{-}\gamma_1}{T}p (\ce, \cs, T)\right), \label{eq:fLi plus H}\\
		\overline{j}^\Li_- & = \ce^{\alphaceLi + \gamma_2 \left( \alpha_{\phie} + \frac{\mu ({x},T)}{ T} \right)} \cs^{\alphacsLi  - \gamma_2 \frac{\lambda_{\min} ({x},T)}{T}} (\csmax - \cs)^{\betacs + \gamma_2 \frac{\lambda_{\max}({x},T)}{T}} \nonumber \\
		&\qquad \times H \left (\frac {-\gamma_2} T (\phis - \phieLi ) \right ) \exp \left( \frac{\gamma_2}{T}p (\ce, \cs, T)\right) , \label{eq:fLi minus H}
		\end{align}
	}
	{and} $H$ is a {bounded} smooth cut-off function of the exponential:
	\begin{equation}
	H(s) = \begin{dcases}
	\exp (s) , & s \le s_\infty, \\
	\zeta(s), & s > s_\infty,     \\
	\end{dcases}
	\end{equation}
	where $s_\infty$  is {an} arbitrarily large but fixed cut-off value, and {$\zeta$ is such that $\zeta' > 0$}. 
\end{hyp}
\begin{rem}\label{rem:cut-off by H} {Notice that \eqref{eq:fLi plus H}-\eqref{eq:fLi minus H} is \eqref{eq:fLi plus}-\eqref{eq:fLi minus} when $H (s) = \exp(s)$.}
\end{rem}

This modification allows {us to} prove Theorem \ref{prop:blow up truncated potential} {(see Section \ref{sec:proof of the blow up with truncated potential})}.

\subsubsection{Truncated temperature behaviour}

{Due to the the delicate interconnectedness  of the different terms in $F_T$ it is too difficult to prove a global existence theorem. Nonetheless, since many authors consider {that the temperature is constant in the cell (see {\cite{chaturvedi2010algorithms,Farrell2005,Ranon2014}})}, we will allow ourselves a substantial simplification on the structural assumption for {$F_T$}, in order to obtain the global uniqueness result {{avoiding the appearance of} possible blow-up phenomena, which {are} related with the potential of Lithium batteries for thermal runaway and explosion under high temperature operation (see \cite{smith+wang2006diffusion+limitation+lithium}).} 

\begin{hyp} \label{hyp:cutoff FT} Assume that $F_T$ is linear in $T$ 
\begin{align}
F_T &= B_T(\ce,\cs, \phie, \phis) + T A_T (\ce, \cs, \phie, \phis),
\end{align}
and consider the that $B_T$ is a nonnegative bounded function $B_T \in [0,\overline B_T]$ and that $A_T$ is bounded $A_T \in [\underline A_T, \overline A_T]$, where $\overline B_T, \underline A_T, \overline B_T \in \mathbb R$ are constant numbers.
\end{hyp}

{This assumptions allow us to prove the result of global existence of solutions given in Theorem \ref{prop:global existence truncated} (see Section \ref{sec:proof of the blow up with truncated potential and temp}).}

{
\begin{rem}
	The case $F_T \equiv 0$ is a particular case satisfying Assumptions \ref{hyp:cutoff FT}, including the case of $T$ just following Newton's law of cooling (if $\alpha_T \neq 0$) or $T \equiv T_0$, constant temperature (if $\alpha_T = 0$).
\end{rem}
}
\section{Green operator for the semilinear elliptic system \eqref{eq:model phie}-\eqref{eq:model phis}. Proof of Proposition \ref{prop:resolvent elliptic} and \ref{prop:resolvent elliptic Lipschitz}} \label{sec:existence of Green operator}

The idea of the proof {of Proposition \ref{prop:resolvent elliptic}} is to follow the idea for ``low overpotentials'' in \cite{Ramos:2015:lithiumionbatteries}, but replacing Lax-Milgram's theorem by the Brézis theory of ``pseudo monotone operators'' (see \cite{brezis1968equations+inequations+dualite}). Under some special assumptions, this result is sometimes referred to as  {Minty-Browder}'s theorem \cite{minty1963monotonicity,browder1965monotone+operators+convex+sets} (see, e.g., \cite{renardy+rogers2006introduction+pdes}). Uniqueness up to a constant was already proved in \cite{Ramos:2015:lithiumionbatteries}.\\

Let us define, for a {given} $(\ce, \csB, T) \in K_X$, the following function:
\begin{align*}
\eta_0 (x) &= - \alpha_{\phie} T f_{\phie} (c_e (x)) - U(x, \ce(x), \csB(x), T), \qquad \forall x \in \J,
\end{align*}
which corresponds to $\eta|_{\phis - \phieLi = 0}$. Since $(\ce, \csB, T)$ {is} known we can {define}, for $x \in [0,L]$ and $\Phi \in \mathbb R${,} 
\begin{equation}
	\underline j^\Li (x, \Phi ) =  \begin{dcases}
	\overline j^\Li (x, \ce(x), \csB(x), T, \Phi + \eta_0(x) ) & x \in \J, \\
	0 & x \in (L_1, L_1 + \delta){.}
	\end{dcases} 
\end{equation}
Notice that
\begin{equation} \label{eq:relation between underline j and j}
\underline j^\Li (x, \phis(x) - \phieLi(x)) = j^{\Li}(x,  \ce(x),  \csB(x),  \phie(x),  \phis(x) ,  T).
\end{equation}
We also define
\begin{align}
j^{\Li}_0 (x) = \underline j^\Li (x, 0) \label{eq:defn j Li zero}
\end{align}
which corresponds to $j^{\Li}|_{\phis - \phieLi = 0}$.

\begin{rem}
	Given $(\ce, \csB, T) \in K_X$, {due to \eqref{eq:jLi continuity} and \eqref{eq:U continuity} we have} $j^{\Li}_0 \in L^\infty (0,L) \cap \mathcal C(\overline \J)$.
\end{rem}

\begin{proof}[Proof of Proposition \ref{prop:resolvent elliptic}]
	We rewrite \eqref{eq:weak formulation phie}--\eqref{eq:weak formulation phis} in terms of $\phieLi$, defining $\tilde \kappa(x) = \kappa(\ce (x), T) \in \mathcal C([0,L])$, as
	\begin{align*}
		\int_0^L \tilde \kappa \frac{ \partial \phieLi}{\partial x }  \der \phietest x - \int_{\J} j^{\Li} \phietest \dx &= 0, \\
		\int_{\J} \sigma \frac{\partial \phis}{\partial x} \der  \phistest x \dx +  \int_{\J} j^{\Li} \phistest \dx&= - {\frac{I}{A}} (\phistest (L) - \phistest(0)), \qquad {\forall (\phietest, \phistest) \in \XPhi}. 
	\end{align*}
	Adding both equations we obtain that {\eqref{eq:weak formulation phie}--\eqref{eq:weak formulation phis} is equivalent to}
	\begin{equation} \label{eq: T}
			\int_0^L \tilde \kappa \frac{ \partial \phieLi}{\partial x }  \der \phietest x\dx + \int_{\J} \sigma \frac{\partial \phis}{\partial x} \der \phistest x \dx+ \int_{\J} j^{\Li} (\phistest - \phietest)\dx =  - { \frac I A} {(\phistest (L) - \phistest(0))}, \quad {\forall (\phietest, \phistest) \in \XPhi} .
	\end{equation}
	Let us define, for $x \in [0,L]$ and $\Phi \in \mathbb R$,
	\begin{equation} \label{eq:defn j Li hat} 
	 \widehat {j^\Li} (x, \Phi)=  {\underline j^{\Li} (x, \Phi)} -j^{\Li}_0 (x).
	\end{equation} 
	{Notice that} $ \widehat {j^\Li} (x, 0) = 0$. We can 	rewrite {\eqref{eq: T}} as
	\begin{align*}
		\int_0^L \tilde \kappa \frac{ \partial \phieLi}{\partial x }  \der \phietest x \dx&+ \int_{\J} \sigma \frac{\partial \phis}{\partial x} \der \phistest x\dx + \int_{\J} \widehat{j^{\Li}}{(x, \phis- \phieLi)}  (\phistest - \phietest)\dx \\
		& =   -  {\frac I A} {(\phistest (L) - \phistest(0))} - \int_{\J} j^{\Li}_0 {(x)}  (\phistest - \phietest) {\dx}.
	\end{align*}
	Let us define the operator {$A_1 : X_\Phi \to X_\phi^*$} by
	$$
		\langle A_1 (\phieLi,\phis), (\phietest,\phistest) \rangle  = \int_{\J}  \widehat {j^\Li} {(x, \phis - \phieLi) }  (\phistest - \phietest)\dx,
	$$
	{for all $(\phieLi, \phis), (\phietest, \phistest) \in \XPhi$.}
	Since $H^1(\J) \subset \mathcal C(\overline{\J})$ and {$(\phieLi, \phis) \in \XPhi \to { \widehat {j^\Li} }(x, \phis - \phieLi) \in \mathcal C(\overline{\J})$} is bounded and continuous {(due to \eqref{eq:jLi continuity} and \eqref{eq:U continuity})} we have that 
	$A_1 :\XPhi \to \XPhi^*$
	is bounded continuous.	
	{Furthermore}, {applying \eqref{eq:relation between underline j and j}, \eqref{eq:defn j Li hat} and Remark \ref{rem:monotonicity F Li} {(due to Assumption \ref{hyp:regularity flux general})}, we can show that} $A_1$ is a monotone operator since, {for all $(\phieLi, \phis), (\widetilde \phieLi, \widetilde \phis) \in \XPhi$ we have that}
	\begin{align*} \nonumber
		\langle A_1 &(\phieLi,\phis ) - A_1 (\widetilde{\phieLi},\tilde{\phis}), (\phieLi,\phis) - (\widetilde{\phieLi}, \widetilde{\phis}) \rangle  \\
		&= { \int_{\J}  \left( \widehat {j^\Li} {(x, \phis - \phieLi) } - \widehat  {j^\Li} {(x, \widetilde \phis - \widetilde \phieLi) } \right) {\left (\phis - \phieLi - ( \widetilde{\phis }- \widetilde {\phieLi} ) \right)}\dx }\\
		& = { \int_{\J}  \left(  \underline j^\Li {(x, \phis - \phieLi) }  -  \underline  j^\Li {(x, \widetilde \phis - \widetilde \phieLi) }\right) {\left (\phis - \phieLi - ( \widetilde{\phis }- \widetilde {\phieLi} ) \right)}\dx }\\ 
		&= { \int_{\J}  \left(	\overline  j^\Li {(x, \ce, \csB, { \phis -  \phieLi } +\eta_0) }  -  \overline  j^\Li {(x, \ce, \csB, T, \widetilde \phis - \widetilde \phieLi +\eta_0) } \right) }\\
		&\qquad \qquad	{\times {\left (\phis - \phieLi - ( \widetilde{\phis }- \widetilde {\phieLi} ) \right)}\dx }\\
		&{=} \int_{\J} F^\Li (x, \ce (x), \cs(x), T , \eta (x) ,  \widetilde \eta (x)  )  \left (\phis - \phieLi - ( \widetilde{\phis }- \widetilde {\phieLi}) \right)^2 \dx \nonumber
	\end{align*}
	{for all $(\phieLi, \phis), (\phietest, \phistest) \in \XPhi$, due to {\eqref{eq:coercivity} (which is true due {to} Assumption \ref{hyp:regularity flux general}),}} where 
	\begin{align*}
		\eta (x) &= \phis (x) - \phieLi (x) - \alpha_{\phie} T {f_{\phie}} (\ce) - U({x},\ce, \cs, T), \\
		\widetilde  \eta (x) &= \widetilde {\phis} (x) - \widetilde {\phieLi} (x) - \alpha_{\phie} T {f_{\phie}} (\ce) - U({x},\ce, \cs, T).
	\end{align*}
	Therefore, {for all $(\phietest, \phistest) \in \XPhi$,}
	\begin{equation}
		\langle A_1 (\phis, \phieLi) - A_1 (\widetilde{\phis}, \widetilde{\phieLi}), (\phis, \phieLi) - (\widetilde{\phis} , \widetilde{\phieLi}) \rangle
		\ge C\int_{\J}\left (\phis - \phieLi - ( \widetilde{\phis }- \widetilde {\phieLi}) \right)^2\dx,  \label{eq:coercivity of T1}
	\end{equation}
	where 
	{$$
		C = C\left(\ce, \cs, T, \phis - \phieLi, ( \widetilde{\phis }- \widetilde {\phieLi}) \right) = \min_{x\in\J} \int_{\J} F^\Li (x, \ce (x), \cs(x), T , \eta (x) ,  \widetilde \eta (x)  ) > 0 {.}
	$$}
	\\
	Let the operator {$\mathcal A:\XPhi \to \XPhi^*$} be defined by:
		\begin{equation} \label{eq:73}
			\langle {\mathcal A}(\phis, \phieLi) , (\phistest, \phietest) \rangle = 	\int_0^L \tilde \kappa \frac{ \partial \phieLi}{\partial x }  \der \phietest x \dx + \int_{\J} \sigma \frac{\partial \phis}{\partial x} \der \phistest x \dx + \langle A_1 (\phis, \phieLi) , (\phistest, \phietest) \rangle.
		\end{equation}
	Then ${\mathcal A}$ is a bounded, continuous, monotone operator. Moreover, it is coercive 
	due to the Poincaré-Wirtinger inequality $\| \phieLi \|_{L^2(0,L)} \le C \| \nabla \phieLi \|_{L^2 (0,L)}$ {and \eqref{eq:coercivity of T1}}. 
	Hence, there exists a unique solution of the system \eqref{eq:model phie}-\eqref{eq:model phis}{, due to the {Minty-Browder} theorem}.
\end{proof}

\begin{rem}
Of course if $(\phie, \phis)$  is a solution of the system \eqref{eq:model phie}-\eqref{eq:model phis} and $C$ is a constant then $(\phie + C, \phis + C)$ is also a solution. However, there exists only one solution in $\XPhi$.
\end{rem} 

\begin{rem}
	The main part of the proof above was to apply the monotonicity of $j^\Li$ with respect to $\phis-\phieLi$. The idea behind this monotonicity method has to do with the convexity of the associated energy functional.
\end{rem}

We have so far proved that {the map $G_\phi$ given by \eqref{eq:defn Gphi} is well-defined}. We prove now, applying the Implicit Function Theorem, that this map is $\mathcal C^1$.

\begin{proof}[Proof of Proposition \ref{prop:resolvent elliptic Lipschitz}]
	We will apply the implicit function theorem for {the Banach space-valued {mapping}} $F : \widehat X \times \widehat Y \to \widehat Z$, to solve for an operator {$\widehat G_\phi: { U  \subset }\widehat X \to \widehat Y$ in an expression of the form
	\begin{equation} \label{eq:implicit formulation G}
		F(\widehat x, \widehat G_\phi (\widehat x) ) = 0, \qquad {\textrm{for all } \widehat x \in U.} 
	\end{equation}}
	The choice of functional spaces will be
	\begin{equation} 
	\widehat X = \mathcal C_{> k_0}([0,L]) \times \mathbb R \times K_X, \qquad \widehat Y = \XPhi, \qquad \widehat Z = \XPhi^*,
	\end{equation} 	
	{with}{ 
	\begin{equation} 
		\mathcal C_{> \kappa_0}([0,L]) = \{  \tilde \kappa \in \mathcal C([0,L]) : \tilde \kappa > \kappa_0  \}, \quad \textrm{for some }\kappa_0 > 0.
	\end{equation}
	} 
	{We will then check that
		\begin{equation} \label{eq:Gphi and hat Gphi}
		G_\phi (\ce, \csB, T, I ) = \widehat G_\phi (\kappa(\ce,T), I, \ce, \csB, T) + (\alpha_{\phie} T f_{\phie} (\ce),0),
		\end{equation}
	due to the definition of $\phieLi$ (see \eqref{eq:definition phieLi}).}
	Notice that $\widehat X$ is an open set of a Banach space, and therefore we can consider the Implicit Function Theorem (see, e.g., \cite{lang2012fundamentals}) in this setting. We will use the notation
	\begin{equation*}
	\widehat x = ({\tilde \kappa}, I, \ce, \csB, T) { \in \widehat X}, \qquad \widehat y = (\phieLi,\phis) = (u,v) { \in \widehat Y.}
	\end{equation*}
	We consider the maps
	$
		A_1, A_2, A_3, A_4: \widehat X \times \widehat Y    \to \widehat Z, 
	$
	given by
	\begin{align*} 
		{ A_1 (\widehat x, \widehat y) (\phietest,\phistest)} &= A_1 (\widehat x, u) (\phietest) = \int_0^L {\tilde \kappa} (x) u ' (x) \phietest ' (x) \dx  \\
		{ A_2 (\widehat x, \widehat y) (\phietest,\phistest)} &= A_2 (v) (\phistest)= \int_{\J } \sigma (x)  v ' (x) \phistest ' (x) \dx  \\
		 \overline \eta_0(x,\widehat x (x)) &= - \alpha_{\phie} T f_{\phie} (c_e (x)) - U(x, \ce(x), \csB(x), T)\\
		A_3 (\widehat x, {\widehat y}) (\phietest, \phistest)&= \int_{\J } \overline{j}^\Li (x, \widehat x (x),  v(x) - u(x) + \overline \eta_0(x,\widehat x (x))) \cdot (\phistest (x) - \phietest (x)) \dx \\
		{ A_4 (\widehat x, \widehat y) (\phietest,\phistest)} &= A_4 (I) (\phistest) =  {\frac I A} (\phistest (L) - \phistest(0)) \\
		F &= A_1 + \cdots + A_4. 
	\end{align*}
	Our definition of weak solution is precisely
	\begin{equation}
		F (\widehat x, \widehat y) = 0_{\widehat Z}.
	\end{equation}
	The function $A_4$ is linear and continuous, therefore $C^\infty$. It is automatic to see that
		$$D_{\widehat y} A_4 = 0. $$
		On the other hand, $A_1, A_2$ are multilinear and continuous, and therefore of class $\mathcal C^1$. In particular, for $\widehat y = (u,v), \bar{\widehat y} = (\bar u , \bar v) \in X_\Phi$ we have
	\begin{align*}
		D_{u} A_1 (\tilde \kappa, u) (\bar u) &= A_1 (\tilde \kappa, \bar {u}) \\
		D_{v }A_2 (u)(\bar v) &= A_2 (\bar v) 
	\end{align*}
	Since $\partiald {\overline j^\Li}{\eta}$ is of class $C^1$ {(see {\eqref{eq:jLi continuity} in} Assumption \ref{hyp:regularity flux general})}, then $A_3$ is also of class $C^1$ and
	\begin{equation*}
		D_{(u, v)} A_3 (\widehat x, u,v) (\bar {u},\bar {v}) (\phietest,\phistest) = \int_{\J} g (\bar v  - \bar u) (\phistest - \phietest)\dx .
	\end{equation*}
	where, for $x \in \J$,
	\begin{equation*}
		 g(x) = \partiald {\overline{j}^\Li}{\eta} (x, \widehat x (x),  v(x) - u(x) + \overline \eta_0(x,\widehat x (x))).
	\end{equation*}
	Let $\widehat x ^0 = (\tilde \kappa, \ce, \csB, T) \in \widehat X$, and let  $\widehat y^0 = (u^0,v^0) = (\phieLi,\phis) \in \XPhi $ be the solution found in Proposition \ref{prop:resolvent elliptic}. Then
	$
	 g(x)  > 0
	$
	is in $\mathcal C(\overline \J)$. Therefore, there exists a constant $g_0$ such that $g(x) \ge g_0 > 0$ in $\J$.  Then 
	\begin{equation*}
		 D_{(u,v)} {F} (\widehat x^0, \widehat y^0): (\bar {u},\bar {v}) \in \XPhi \to \XPhi^*
	\end{equation*}
	understood {as} a bilinear form
	\begin{align*}
		G(  (\bar {u},\bar {v}),(\phietest, \phistest)  ) = \int_0 ^L \tilde \kappa \bar  {u} ' \phietest' + \int _{\J} \sigma \bar {v}' \phistest' + \int_{ \J } g (x) (\bar {u} - \bar {v}) (\phistest - \phietest) {,}
	\end{align*}
	is continuous and coercive in $\XPhi \times \XPhi$. Therefore, by Lax-Milgram's theorem $D_{(u,v)} F(\widehat x^0, \widehat y^0)$ is bijective. Then, by the {Implicit Function Theorem} applied to $F$, there exists a unique $C^1$ function {$\widehat G_\phi:  U  \to \XPhi $}, defined in a neighbourhood {$U$} of $\widehat x^0$ in $\widehat X$, such that {\eqref{eq:implicit formulation G}} {holds}. 
	\\	
	Since $(\ce, T) \in H^1(0,L) \times \mathbb R \mapsto \kappa (\ce, T) \in \mathcal C([0,L])$ is also $\mathcal C^1$ {(the function $\kappa$ is of class $\mathcal C^2$ due to Assumptions \ref{hyp:regularity of coefficients}) we have that the map
	$$(\ce, \csB, T,I) \in K_X \times \mathbb R \, \overset{J} \longmapsto \, (\kappa (\ce, T), I, \ce, \csB, T) \in \widehat X$$
	is $C^1$. Let us choose a point $(\ce^0, \csB^0,T^0, I^0) \in K_X \times \mathbb R$, and let $ \widehat x_0 = (\kappa (\ce^0, T^0), I^0, \ce^0, \csB^0, T^0) \in \widehat X$. Let $U \subset \widehat X$ be a {suitable} neighbourhood of $\widehat x_0$ so that $\widehat G_\phi : U \to \XPhi$ is defined {satisfying \eqref{eq:implicit formulation G}}. Tak{ing} the neighbourhood of $(\ce^0, \csB^0,T^0, I^0)$ given by $V = J^{-1} (U) \subset K_X \times \mathbb R${, the} composition
	$$		(\ce, \csB, T,I) \in V \, \overset J \longmapsto \, (\kappa (\ce, T), I, \ce, \csB, T) \in U \, {\overset{ \widehat G_\phi} \longmapsto} \, (\phieLi, \phis) \in \XPhi$$
	is $\mathcal C^1$. We finally consider the following translation (which is also of class $\mathcal C^1$)
	\begin{eqnarray*}
		\tau : V \times \XPhi &\longrightarrow& H^1(0,L) \times H^1(\J) \\
		(\ce, \csB, T,I,\phieLi, \phis) & \longmapsto & (\phie, \phis)= (\phieLi + \alpha_{\phie} T f_{\phie} (\ce),\phis).
	\end{eqnarray*}
	Due to {the} uniqueness {(up to a constant)} {result} we proved in Proposition \ref{prop:resolvent elliptic}, 
	\begin{eqnarray*}
			\tau \circ (Id, \widehat G_\phi \circ J): V \subset K_X \times \mathbb R &\longrightarrow& H^1(0,L) \times H^1(\J)
	\end{eqnarray*}
	is the map $G_\phi|_{V}$ (as constructed in \eqref{eq:defn Gphi} through Proposition \ref{prop:resolvent elliptic}). Thus, $G_\phi$ is of class $C^1$ in a neighbourhood of $(\ce^0, \csB^0,T^0, I^0)$. Since this argument holds over any point $(\ce^0, \csB^0,T^0, I^0) \in K_X \times \mathbb R$, we have shown that $G_\phi$ is of class $\mathcal C^1$ over $K_X \times \mathbb R$.  This implies that $\widetilde G_\phi$ is also $\mathcal C^1$ {and} {it} concludes the proof.
	}
\end{proof}

Therefore, {as an immediate consequence of Proposition~\ref{prop:resolvent elliptic Lipschitz}}{,} we have the following lemma {in terms of the space of piecewise continuous functions $\mathcal C_{\rm part}$ defined by \eqref{eq:defn piecewise continuous}:}
	\begin{lem} \label{lem:G phi locally Lipschitz continuous}
		Let $I \in \mathcal C([0,t_0])$ then ${ \widetilde G_{\phi,t} }: \mathcal C([0,t_0], K_X) \to \mathcal C([0,t_0], K_Z)$ is locally Lipschitz continuous. If $I \in \mathcal C_{\rm {part}}([0,t_0])$ then ${ \widetilde G_{\phi,t} } : \mathcal C([0,t_0], K_X) \to \mathcal C_{\rm {part}}([0,t_0], K_Z)$ is locally Lipschitz continuous.
	\end{lem}

	
	\section{Regularity of the Green operators $G_{\ce,t}$ and $G_{\csB,t}$}
	
	The regularity of $G_{\ce,t}$ is a well-known property:
	\begin{equation*}
		G_{\ce, t_0} : L^2((0,t_0) \times (0,L)) \to \mathcal C ([0,t_0];H^1(0,L)).
	\end{equation*}
	This operator has a nice representation formula
	\begin{equation*}
		(G_{\ce, t_0} f ) (t) = S (t) c_{e,0} + \int_0^t S(t-s) f(s)\ds, 
	\end{equation*}
	where $S(t) u_0$ is the solution of 
	\begin{equation*}
		\begin{dcases}
			\frac{\partial u }{\partial t} - \partiald {} {x} \left( D_{\rm e} \partiald {\ce}{x}  \right) = 0, & (0,L) \times \mathbb R, \\
			\frac{\partial u}{\partial x} = 0, & \{0,L\} \times \mathbb R , \\
			u(0) = u_0, & t = 0. 
		\end{dcases}
	\end{equation*}
	A number of properties can be easily derived from this expression. {For instance,} the continuous dependence with respect to the data:
	\begin{equation*}
		\| G_{\ce, t_0} f - G_{\ce,t_0} g \|_{\mathcal C([0,t_0]; H^1(0,L))} \le \rho (t_0) \| f - g \|_{L^2(0,t_0; H^1(0,L))},
	\end{equation*}
	where $\rho$ is and increasing, continuous function such that $\rho (0) = 0$. 
	\\
	The term $G_{\csB, t}$ is a little trickier. First we recall (see \cite{Arendt+al:2013vectorvaluecauchy,Nittka:2014:inhomogeneous-neumann}) that, for any $q > 2$:
	\begin{equation*}
		G_{\cs,R,t_0}  : \mathcal C ((0,R)) \times L^q(0,t_0) \to \mathcal C ([0,t_0]\times [0,R] ),
	\end{equation*}
	and
	\begin{equation*}
		\| G_{\cs,R,t_0} (u_0,g) \|_{L^\infty (0,t_0;L^\infty (B_R))} \le C ( \|u_0\|_{L^\infty (B_R)} + \| g \| _{L^q(0,t_0)}) .
	\end{equation*} 
	Therefore, due to the linearity of the equation, for $x,y$ in the same connected component of $\J$, we have that:
	\begin{equation*}
		\| G_{\cs,R,t_0 }(u_0,g) -  G_{\cs,R,t_0} (v_0,h) \|_{L^\infty (0,T;L^\infty (B_R))} \le C ( \|u_0 - v_0 \|_{L^\infty (B_R)} + \| g - h \| _{L^q(0,T)}) .
	\end{equation*} 	
	This solves the problem of continuity of $G_{\cs} g$ with respect to $x$ via the continuous dependence of the operator. Since $c_{s,0}$ is continuous, working in each component we can prove directly that
	\begin{equation*}
		G_{\csB,t_0}  : \mathcal C(\overline \J; L^q (0,t_0)) \to \mathcal C ( \overline \J \times [0,t_0]),
	\end{equation*} 
	is Lipschitz continuous. Furthermore it {is easy} to check that
		\begin{align*}
		G_{\csB,t_0} & : \mathcal C(\overline  \J \times [0,t_0]) \to \mathcal C (\overline \J \times [0,t_0]).
		\end{align*}
	We also have the following {time} estimate, for $t_0 \ge 0$,
		\begin{align*}
		\| G_{\csB,t_0} g  &-  G_{\csB,t_0} h \|_{L^\infty ([0,t_0] \times \J)} \le C  t_0^{\frac 1 q} \| g - h \|_{L^\infty ({(}0,t_0{)} \times \J)}.
		\end{align*} 	
	By defining the vectorial Green operator for the evolutionary part
	\begin{equation*}
		 \mathbf G _t = (G_{\ce,t}  G_{\cs,t} , G_{T,t} ) : \mathcal C ([0,t]; Y) \to \mathcal C ([0,t]; X),
	\end{equation*}
	due to the previous results{,} we have 
	\begin{equation} \label{eq:Gt Lipschitz constant}
		\| \mathbf G_t \mathbf y - \mathbf G_t \mathbf {\hat y} \|_{\mathcal C ([0,t]; X)} \le \rho (t) \| \mathbf y - \mathbf{\hat y} \| _{\mathcal C ([0,t]; Y)},
	\end{equation}
	where $\rho$ is a continuous function such that $\rho(0) = 0$.

	\section{Proof of Theorem \ref{thrm:well posedness model} \label{sec:well-posedness}}
	\label{sec:proof of theorem 1}
	
	\begin{proof}[Proof of Theorem~\ref{thrm:well posedness model}]
	First, let us assume that $I$ is continuous. { Let us define the function}
	\begin{align*}
		\mathbf f(\ce,\csB,{\phie}, \phis, T) = \bigg( &  \alpha_e (x) N_{j^\Li} (\ce, \csB ,  {\phie}, \phis,T ),  \\
		& \alpha_{\phis} N_{j^\Li} (\ce, \csB ,  {\phie}, \phis,T ), \\
		& - h A_s (T - \Tamb) + F_T \bigg) .
	\end{align*}
	{It is clear that $\mathbf f : K_X \to Y$ {is} locally Lipschitz continuous {(due to the definition and regularity of $N_{j^\Li}$ and $F_T$)}. {Let us define}, for any $t>0$, 
		$$\mathbf f_t : \mathcal C( [0,t] ; K_X ) \to \mathcal C([0,t] ; Y),$$
	by $\mathbf f_t (\mathbf x ) (s)  = \mathbf f (\mathbf x(s))$ for $s \in [0,t]$. This operator is also locally Lipschitz continuous.}
	Let $\mathbf x = (\ce, \csB, T)$. Then we can rewrite {the fixed point problem  \eqref{eq:solution as fixed point in csB} (which was our definition of weak-mild solution of \eqref{eq:model ce}-\eqref{eq:model T})} as
	\begin{equation} \label{eq:solution as fixed point in vector form}
		\mathbf x  = \mathbf G_t \circ { \mathbf f_t \circ  \widetilde G_{\phi,t} ( \mathbf x ) } .
	\end{equation}
	\\
	Since ${\mathbf f_t \circ \widetilde G_{\phi,t}} : \mathcal C([0,t_0], K_X) \to \mathcal C([0,t_0], Y)$ is locally Lipschitz continuous {(due to Lemma~\ref{lem:G phi locally Lipschitz continuous})}, we can set a bounded neighbourhood $\mathbf U \subset K_X$ around $\mathbf x(0) = (\ce(0), \csB(0), T (0)) \in K_X$, and a bounded set $\mathbf V \subset Y$, such that the composition
	$${\mathbf f_t \circ \widetilde G_{\phi,t_0}}  : \mathcal C([0,t_0], \mathbf U) \to \mathcal C([0,t_0], \mathbf V)$$
	 is globally Lipschitz continuous. Due to the continuity of $\mathbf G_t$ and the fact that $\mathbf G_0 \equiv \mathbf x(0) \in K_X$, there exists $t_1 > 0$ such that
	$$		\mathbf G_{t_1}: \mathcal C([0,t_1], \mathbf V) \to \mathcal C([0,t_1], \mathbf U).$$
	Therefore, due to \eqref{eq:Gt Lipschitz constant}, we obtain that, for $t_2 = \min\{t_0, t_1\} > 0$ we have that
	$$\mathbf G_{t_2} \circ {\mathbf f_t \circ \widetilde G_{\phi,t_2}} :  \mathcal C([0,t_2], \mathbf U) \to \mathcal C([0,t_2], \mathbf U)$$
	is a contracting map. Then, we can apply the Banach fixed point theorem to find a unique solution of problem \eqref{eq:solution as fixed point in vector form}. If $I$ is $\mathcal C_{\rm {part}}$ then one can simply paste the mild solutions from the different time partitions. In this sense there exists a unique ``{piecewise} weak-mild solution''.
	\\
	Applying classical results, there exists a maximal existence time $\tend \le \tendI$.  If $\tend < \tendI$ then $d ( \mathbf x (t), \partial K_X ) \to 0 $ or $\| \mathbf x (t) \|_X \to + \infty$ as $t \to \tend$. This is equivalent to \eqref{eq:solutions to boundary} and the proof is complete.
	\end{proof}
	
	\section{Proof of Theorem \ref{thrm:blow up behaviour}}
	
	There are a number of papers studying the semilinear equation $u_t - \Delta u = f(u)$ with Robin type boundary conditions, and its eventual possible blow up, depending on the growth of $f$. Some of the most classical results are due to Amann \cite{Amann:1987abstract} {(for some more recent works, {see} e.g.,  \cite{pao2012nonlinear+parabolic})}.

		Let us assume that the solution $(\ce, \csB, \phie, \phis, T)$ is defined in $[0,t_0)$, where $0 < t_0 < \tendI$, and assume that \eqref{eq:reasonable condition blowup} does not hold as $t \nearrow t_0$. We will show that \eqref{eq:solutions to boundary} does not hold, and therefore $\tend > t_0$. We can think of the problem written in the following way
			\begin{equation*}
			\frac{\partial \ce}{\partial t} -{  \frac{\partial }{\partial x } \left(D_{\rm e} \frac{\partial \ce}{\partial x}\right)}+ C_1 \ce^{\alphaceLi + \gamma_2 \left( \alpha_{\phie} + \frac{\mu (\hat T)}{\hat T} \right)} = C_2 \ce^{\alphaceLi - \gamma_1 \left( \alpha_{\phie} + \frac{\mu (\hat T)}{\hat T} \right)} \ge 0,
			\end{equation*}		
			where $C_1$ and $C_2$ are functions which we will show can be estimated. We can define
			\begin{align*}
			\overline \mu & = \max_{{[0,L]\times}[0,t_0]} \left( \alpha_{\phie} + \frac{\mu ({x},T(t))}{ T(t)} \right),  \\
			\underline \mu & = \min _{{[0,L]\times}[0,t_0]} \left( \alpha_{\phie} + \frac{\mu ({x},T(t))}{T(t)} \right) \ge 1,\\
			\overline C_1 &= { \max_{x \in \J}{(\delta_1 + \delta_2 )} } {\left( \max_{[0,L]}\alphaceeq \right) }   \exp \left( { \frac{\gamma_1 + \gamma_2}{\min _{[0,t_0]}T(t)}  (\| \phistest - \phie \|_{L^\infty } +  \|p \|_{L^\infty} ) } \right) .
			\end{align*}
			Notice that, since $\csmax < 1$, we can conclude that $0 \le  C_1, C_2 \le \overline C_1$. Finally, we can define some increasing continuous functions $\beta_1, \beta_2$ such that $\beta_1 (0) = \beta_2 (0) = 0$, {with} $\beta_2$ Lipschitz continuous and 
			\begin{align*} 
			\max  \{ s ^{\alphaceLi - \gamma_1 \underline \mu }, s^{\alphaceLi - \gamma_1 \overline \mu }\} \le \beta_1 (s) \le 1 + s , \\
			\beta_2 (s) \ge \max \{ 
			s^{\alphaceLi + \gamma_2 \overline \mu} , s^{\alphaceLi + \gamma_2 \underline \mu}  
			\},
			\end{align*}
			so that we can construct the supersolution $\overline \ce$ and subsolution $\underline {\ce}$ defined as solutions{, respectively, of}
			\begin{align*}
			&\begin{dcases}
			\frac{\partial \overline  \ce}{\partial t} - {  \frac{\partial }{\partial x } \left(D_{\rm e} \frac{\partial\overline \ce}{\partial x} \right) } = \overline C \beta_1 (\overline {\ce }) & (x,t) \in (0,L) \times (0,\tend), \\
			\overline  \ce (0) = c_{e,0} & t= 0, \\
			\partial_n \overline  \ce = 0 & x \in \{0,L\},
			\end{dcases}\\
			&\begin{dcases}
			\frac{\partial \underline \ce}{\partial t} -{  \frac{\partial }{\partial x } \left(D_{\rm e} \frac{\partial\underline \ce}{\partial x} \right) }  + \overline C \beta_2 (\underline {\ce }) = 0 & (x,t) \in (0,L) \times (0,\tend),\\
			\underline \ce (0) = c_{e,0}& t= 0, \\
			\partial_n \underline \ce = 0 & x \in \{0,L\}.
			\end{dcases}
			\end{align*}
			Using the conditions on the exponents {given by Assumptions~\ref{hyp:bounds exponents}}{,} we deduce that $\overline \ce, \underline \ce$ are continuous, globally defined in time and
			\begin{equation*}
			 \min_{[0,L] \times [0,t_0]} \underline \ce > 0 \qquad \max_{[0,L] \times [0,t_0] } \overline \ce < +\infty {.}
			\end{equation*}
			We also have that
			\begin{equation*}
			\begin{dcases}
			\frac{ \partial \cs}{\partial t} - {\frac{D_{\rm s}}{r^2}{ \frac{\partial }{\partial r} \left(r^2 \frac{ \partial \cs}{\partial r} \right)}} = 0, & {(x,r,t) \in \Done \times (0,t_0)},\\
			{\frac {\partial \cs} {\partial r} }  + C_3 \cs^{\alphacsLi + \gamma_1 \frac{\lambda_{\min} ({x}, T)}{T}} = \alpha_{\cs } j_{-}^\Li \ge 0, & {r} = \Rs(x), \\
			\cs = c_{\textrm{s},0} & t = 0, \\
			\end{dcases}
			\end{equation*}
			where $C_3$ is a suitable function, such that if we define
			\begin{align*}
			\underline  \lambda_{\min} &= \min_{{[0,L]\times}[0,t_0]}  \frac{ \lambda _{\min} ({x}, T(t)) }  {T(t)} , \\
			\overline   \lambda_{\min} &= \max_{{[0,L]\times}[0,t_0]}  \frac{ \lambda _{\min} ({x},T(t)) }  {T(t)}  \\
			\overline C_3 &= { \max_{x \in \J}{(\delta_1 + \delta_2 )} }  {\left( \max_{[0,L]}\alphacseq \right) }  \left( \max_{\alpha \in \{ \alphaceLi \pm \gamma_{1,2} \underline \mu, \alphaceLi \pm \gamma_{1,2} \overline  \mu\}} \| \ce \|_{L^\infty} ^ {\alpha} \right)  \\
				& \qquad \times  \exp \left( \frac 1 {{\min_{[0,t_0]}}T(t)} { \max\{ \gamma_1 , \gamma_2 \} (\| \phistest - \phie \|_{L^\infty } +  \|p \|_{L^\infty} ) } \right) {,}
			\end{align*}
			then $0 \le C_3 \le \overline C_3$.	Now, let {$\beta_3$ be} a monotone increasing locally Lipschitz continuous function
			such that $ \beta_3(0) = 0$ and 
			\begin{equation*}
			\beta_3 (s) \ge \max \{  s^{\alpha_{c_s} + \gamma_1 \underline \lambda_{\min} } , s^{\alpha_{c_s} + \gamma_1 \overline \lambda_{\min } } \}.
			\end{equation*}
			We construct the subsolution $\underline \cs$, for every $t_0 < + \infty ${,} as the solution of  
			\begin{equation*}
				\begin{dcases}
				\partiald {\underline \cs} t - {\frac{D_{\rm s}}{r^2}{ \frac{\partial }{\partial r} \left(r^2 \frac{ \partial \underline \cs}{\partial r} \right)}} = 0, & (x,y,t) \in {(x,r,t) \in \Done \times (0,t_0)},\\
				\partial_n \underline\cs + \overline C_3 \sigma (\underline \cs) = 0, & r = \Rs (x) {,}\\
				\underline \cs (x,y,t) = \min_{\sigma \in \J} c_{\textrm{s},0}(\sigma,y) & t= 0{,}
				\end{dcases}
			\end{equation*}
			so that, by the comparison principle, we have that $0 < \underline \cs \le \cs$, {(by applying Assumptions~\ref{hyp:bounds exponents})}. Finally{,} if we write
			\begin{equation} \label{eq:problem for csmax-cs}
				\begin{dcases}
				\frac{\partial }{\partial t}(\csmax - \cs) - {\frac{D_{\rm s}}{r^2}{ \frac{\partial }{\partial r} \left(r^2 \frac{ \partial (\csmax - \cs)}{\partial r} \right)}} = 0 & (x,y,t) \in {(x,r,t) \in \Done \times (0,t_0)},\\
				\frac{\partial }{\partial r}  (\csmax - \cs) + C_4  (\csmax - \cs)^{\betacs + \gamma_2 \frac{\lambda_{\max} (T)}{T}} = j_{+}^\Li \ge 0 & r = \Rs(x), \\
				\csmax - \cs = \csmax - c_{\textrm{s},0} & t = 0,
				\end{dcases}
			\end{equation}
			since $0\le C_4 \le \overline C_3$, if we introduce
			\begin{align*}
			\underline  \lambda_{\max} &= \min_{{[0,L]\times}[0,t_0]}  \frac{ \lambda _{\max } ({x},T(t)) }  {T(t)} , \\
			\overline   \lambda_{\max} &= \max_{{[0,L]\times}[0,t_0]}  \frac{ \lambda _{\max } ({x},T(t)) }  {T(t)} 
			\end{align*}
			and define $\beta_4$ as a monotone increasing locally Lispchitz continuous function such that $ \beta_4(0) = 0$ and 
			\begin{equation*}
			\beta_4 (s) \ge \max \left \{  s^{\alpha_{c_s} + \gamma_2 \underline \lambda_{\max } } , s^{\alpha_{c_s} + \gamma_2 \overline \lambda_{\max  } } \right \}
			\end{equation*}
			then, by defining $\tilde \cs$ as the solution of 
			\begin{equation*}
				\begin{dcases}
				\partiald {\tilde \cs} t  - {\frac{D_{\rm s}}{r^2}{ \frac{\partial }{\partial r} \left(r^2 \frac{ \partial \tilde \cs}{\partial r} \right)}} = 0 & (x,y,t) \in \Dthree \times (0,t_0), \\
				\partiald{}{r}  \tilde \cs + \overline C_3 \beta_4( \tilde \cs) = 0 &  |y|=\Rs(x) \\
				\tilde \cs (x,y,t) = \min_{\sigma \in \J} (\csmax - c_{\textrm{s},0}(\sigma,y))& t = 0,
				\end{dcases}
			\end{equation*}
			we arrive to the {definition of the searched} function $\overline \cs = \csmax - \tilde \cs$. We have that $\tilde \cs > 0$ {(due to Assumption~\ref{hyp:bounds exponents})} and {$\tilde \cs$} is a subsolution of \eqref{eq:problem for csmax-cs}. Hence $\tilde \cs \le \csmax - \cs$. Therefore we have that $\cs \le \overline \cs < \csmax$. Putting these two bounding solutions together we conclude
			\begin{equation*}
			0 < \underline \cs \le \cs \le \overline \cs < \csmax.
			\end{equation*}
			Hence \eqref{eq:solutions to boundary} does not hold as $t \nearrow t_0$. Applying Theorem \ref{thrm:well posedness model} {we obtain that} $\tend > t_0$. Therefore, either $\tend = \tendI$ or \eqref{eq:reasonable condition blowup}.
	
\section{Proof of a {global in time existence} result}

\subsection{On the truncated potential case. Proof of Theorem \ref{prop:blow up truncated potential}}

\label{sec:proof of the blow up with truncated potential}

\begin{proof}[Proof of Theorem \ref{prop:blow up truncated potential}]
Assume that \eqref{eq:blow up of T} does not hold as $t \to t_0$ with $0 < t_0 < \tendI$. We can substitute the constants in the proof of Proposition \ref{prop:blow up truncated potential}
\begin{align}
	\overline C_1 &={ \max_{x \in \J}{(\delta_1 + \delta_2 )} }  \left(  \max_{[0,L]
	} \alphaceeq \right)   \exp \left( { \frac {\gamma_1 + \gamma_2} {\min_{[0,t_0]}T}   \|p \|_{L^\infty} } \right), \\
	\overline C_3 &= { \max_{x \in \J}{(\delta_1 + \delta_2 )} } \left( \max_{[0,L]} \alphacsLi \right)\left( \max_{\alpha \in \{ \alphaceLi \pm \gamma_{1,2} \underline \mu, \alphaceLi \pm \gamma_{1,2} \overline  \mu\}}  \| \ce \|_{L^\infty} ^ {\alpha} \right) \| H \|_{L^\infty}  \exp \left( { \frac {\gamma_1 + \gamma_2} {\min_{[0,t_0]}T}  \|p \|_{L^\infty} } \right),
\end{align}
and repeat the argument. We get good bounds for $\ce, \csB$ for $t \in [0, t_0]$, which do not depend on $\| \phis - \phieLi \|_{L^\infty}$. Hence, applying Lemma \ref{lem:G phi locally Lipschitz continuous} we can obtain some estimates of $\phis$ and $\phieLi$ in $[0,t_0]$. Therefore, by Theorem \ref{thrm:blow up behaviour}, we have that $\tend > t_0$. Hence, by contraposition, if $\tend < \tendI$ then \eqref{eq:blow up of T} must hold.
\end{proof}

\subsection{On the truncated temperature case. Proof of Theorem \ref{prop:global existence truncated}}

\label{sec:proof of the blow up with truncated potential and temp}

\begin{proof} [Proof of Proposition \ref{prop:global existence truncated}]
	It is immediate to establish global sub and supersolutions for $T$, which ensure \eqref{eq:blow up of T} does not happen, and hence {we get} the global existence of solutions.
\end{proof}

\section{Final remarks}

\subsection{Mass conservation and derived properties} \label{sec:mass conservation}
We begin by analyzing the compatibility conditions for the elliptic problems with Neumann boundary conditions. The compatibility conditions for the existence of $\phis$ are 
\begin{equation}
\int_0^{L_1} j^{\Li} \dx = {+\frac{I (t)}{A}}, \qquad \int_{L_1 + \delta}^{L} j^{\Li} \dx =  {- \frac{I(t)}{A}}.
\end{equation}
Hence, we deduce that
\begin{equation}
\int_0^L j^\Li \dx = 0.
\end{equation}
This is also the compatibility condition for system \eqref{eq:model phis}.

Let us take a look at the mass balance condition. {From system \eqref{eq:model cs} we} observe that
\begin{equation*}
\frac{\rm d}{\rm d t} \left( \int_0^{L} \ce \dx \right) = \int_0^L \left(\frac{\partial}{\partial t} \ce \right) \dx =   \int_0^L \alphaceeq j^{\Li} \dx  +D_e \left(   \frac{\partial \ce}{\partial x}  (L) {-} \frac{\partial \ce}{\partial x} (0)\right) = 0 .
\end{equation*} 
Thus, {as expected, the mathematical model satisfies that} the total concentration in the electrolyte is constant. Since $\ce \ge 0$ {(we have constructed the convex spaces $K_X$ and $K_Z$ so that the solution of the mathematical model satisfies this condition, which is {consistent with the physics of the problem} since it is a concentration)} this conclusion also implies that $\| \ce \|_{L^1(0,L)}$ is constant. In other words, as a whole, the electrolyte is saturated of Lithium. Any Lithium contribution from an electrode is immediately compensated somewhere else.\\

In fact, if one considers the Lithium concentration {in} the anode and cathode{, for the pseudo two-dimensional (P2D) model considered here,} one gets
\begin{align*}
\frac{ \rm d }{ \rm d t } \left( {4\pi} \int \limits _0^{L_1} \int \limits _0^{{\Rsm}}  \cs r^2 \dr \dx \right) &= {4\pi}\int \limits _0^{L_1} \int \limits _0^{{\Rsm}} \frac{\partial \cs}{\partial t}r^2 \dr \dx = {4\pi} \int \limits _0^{L_1} \int \limits _0^{{\Rsm}}  D_s  \frac{\partial}{\partial r} \left(r^2 \frac{ \partial \cs }{\partial r} \right) \dr \dx \\
&=   {-}{4\pi}\int_0^{L_1}{\Rsm^2 \alphacsLim }  j^\Li  \dx = -4\pi \frac{\Rsm^2 \alphacsLim I(t)}{A}, 
\end{align*}
and, analogously,
\begin{align*}
\frac{ \rm d }{ \rm d t } \left( {4 \pi} \int \limits _{L_1 + \delta}^{L} \int \limits _0^{{\Rsp}} \cs  r^2 \dr \dx \right)  {= +4\pi \frac{\Rsp^2 \alphacsLip I(t)}{A}}.
\end{align*}
If
\begin{equation} 
{\Rsp^2 \alphacsLip = \Rsm^2 \alphacsLim,}
\end{equation}
then {the P2D model satisfies that} any lose {of Lithium in its} anode is instantaneously received by {its} cathode and viceversa. Otherwise, there {would} {appear to be} a net variation in the total amount of Lithium in the solid phase.\\

{The state of the charge at time $t$ of the cell, $\mathrm{ SOC }(t)$, can be then estimated as a normalized average of the
mass of Lithiun in one of the electrodes, typically the negative one (see \cite{Ramos:2015:lithiumionbatteries}):
$$ \mathrm{ SOC }(t)= \frac{3}{L_1 (\Rsm)^3} \int_0^{L_1} \int_{0}^{\Rsm} r^2 \frac{\cs (r,t;x)}{\csmax} \dr \dx .$$}

\subsection{Model validation and identification of parameters}

	In order to validate the model and {identify suitable} parameters, {it is {often} necessary} to compare the solutions of the model with real-world data {and minimize some error functional using suitable parameter identification techniques and minimization algorithms (see, e.g., \cite{fraguela+infante+ramos2013uniqueness+resultation+parameters,infante+molina-rodriguez+ramos2015identifcation+expansion,ivorra+mohammadi+ramos2015improve+initialization}).} {Since} internal properties of the battery such as the spatial distribution of Lithium ions cannot be measured during operation, {it} is standard to measure only the temperature $T$ (on the outside of the battery, but it is a safe assumption that the temperature is spatially homogeneous) and the output voltage $V$ (see, for example, \cite{Birkl2013,Zhang2014}){, which can be estimated by using \eqref{eq:voltage}.}
	
\section{Conclusions}
{In this work, we have shown that the system of equations \eqref{eq:model ce}--\eqref{eq:model T} has a unique solution (under some mild conditions). In the most general setting, only local in time existence result can be shown, and the nature of the possible blow-up is characterized. The most difficult part of the study of an eventual possible blow-up is the structure of the open circuit potential $U$. By considering the model of $U$ given by \eqref{eq:open circuit potential general}, which was proposed in \cite{ramos+please2015arxiv}, we can rule out {some non-physical behaviours as the possible} causes of blow-up, by making some assumptions on the coefficients. Finally, for the sake of mathematical completeness, we provide some extra conditions that allow for global uniqueness in time {and make some comments regarding the conservation of Lithium in the model and the model validation.}}

\section*{Acknowledgements}

The research of {the authors} was partially supported by the {Spanish Ministry of Economy, Industry and Competitiveness under
projects} MTM2014-57113-P {and MTM2015-64865-P}. The research of D. G\'{o}mez-Castro was supported by a FPU
Grant from the Ministerio de Educaci\'{o}n, Cultura y Deporte (Spain). All {the} authors are members of
the Research Group MOMAT (Ref. 910480) of the UCM 


\end{document}